\documentclass[11pt,reqno,letterpaper,final]{amsart}
\usepackage{color}
\usepackage[colorlinks=true, allcolors=blue,backref=page]{hyperref}
\usepackage{amsmath, amssymb, amsthm}
\usepackage{mathrsfs}
\usepackage{mathtools}
\usepackage[noabbrev,capitalize,nameinlink]{cleveref}
\crefname{equation}{}{}
\usepackage{fullpage}
\usepackage[noadjust]{cite}
\usepackage{graphics}
\usepackage{pifont}
\usepackage{tikz}
\usepackage{bbm}
\usepackage[T1]{fontenc}

\usetikzlibrary{arrows.meta}

\usepackage{environ}
\usepackage{framed}
\usepackage{url}
\usepackage[linesnumbered,ruled,vlined]{algorithm2e}
\usepackage[noend]{algpseudocode}
\usepackage[labelfont=bf]{caption}
\usepackage{cite}
\usepackage{framed}
\usepackage[framemethod=tikz]{mdframed}
\usepackage{appendix}
\usepackage{graphicx}
\usepackage[textsize=tiny]{todonotes}
\usepackage{tcolorbox}
\usepackage{enumerate}
\allowdisplaybreaks[1]
\usepackage{stmaryrd}
\usepackage{colonequals}
\usepackage[margin=1in]{geometry}

\usepackage[shortlabels]{enumitem}
\crefformat{enumi}{#2#1#3}
\crefrangeformat{enumi}{#3#1#4 to~#5#2#6}
\crefmultiformat{enumi}{#2#1#3}
{ and~#2#1#3}{, #2#1#3}{ and~#2#1#3}

\DeclareSymbolFont{symbolsC}{U}{pxsyc}{m}{n}
\SetSymbolFont{symbolsC}{bold}{U}{pxsyc}{bx}{n}
\DeclareFontSubstitution{U}{pxsyc}{m}{n}
\DeclareMathSymbol{\medcircle}{\mathbin}{symbolsC}{7}

\crefname{algocf}{Algorithm}{Algorithms}

\crefname{equation}{}{} 
\AtBeginEnvironment{appendices}{\crefalias{section}{appendix}} 

\usepackage[color,final]{showkeys} 

\colorlet{refkey}{orange!20}
\colorlet{labelkey}{blue!30}

\crefname{algocf}{Algorithm}{Algorithms}

\numberwithin{equation}{section}
\newtheorem{theorem}{Theorem}[section]
\newtheorem{proposition}[theorem]{Proposition}
\newtheorem{lemma}[theorem]{Lemma}

\crefname{claim}{Claim}{Claims}

\newtheorem*{question*}{Question}

\theoremstyle{definition}
\newtheorem{definition}[theorem]{Definition}

\newtheorem*{definition*}{Definition}

\theoremstyle{remark}
\newtheorem*{remark}{Remark}

\newlist{enumthm}{enumerate}{1}
\setlist[enumthm]{label=\textup{(\roman*)},ref=\thethm(\roman*)}
\Crefname{enumthmi}{Theorem}{Theorems}

\newcommand{\snorm}[1]{\lVert#1\rVert}
\newcommand{\sang}[1]{\langle #1 \rangle}

\newcommand{\imod}[1]{~\mathrm{mod}~#1}

\newcommand{\mb}{\mathbb}

\newcommand{\mbm}{\mathbbm}
\newcommand{\mc}{\mathcal}
\newcommand{\mf}{\mathfrak}
\newcommand{\mr}{\mathrm}

\newcommand{\ol}{\overline}
\newcommand{\on}{\operatorname}

\newcommand{\wt}{\widetilde}

\newcommand{\poly}{\on{poly}}

\newcommand{\eps}{\varepsilon}

\let\originalleft\left
\let\originalright\right
\renewcommand{\left}{\mathopen{}\mathclose\bgroup\originalleft}
\renewcommand{\right}{\aftergroup\egroup\originalright}

\allowdisplaybreaks

\title{Improved bounds for five-term arithmetic progressions}

\author[A1]{James Leng}
\address{Department of Mathematics, UCLA, Los Angeles, CA 90095, USA}
\email{jamesleng@math.ucla.edu}

\author[A2]{Ashwin Sah}
\author[A3]{Mehtaab Sawhney}
\address{Department of Mathematics, Massachusetts Institute of Technology, Cambridge, MA 02139, USA}
\email{\{asah,msawhney\}@mit.edu}
\thanks{Leng supported by an NSF Graduate Research Fellowship Grant No. DGE-2034835. Sah and Sawhney were supported by NSF Graduate Research Fellowship Program DGE-2141064.}
\begin{document}

\maketitle
\begin{abstract}
Let $r_5(N)$ be the largest cardinality of a set in $\{1,\ldots,N\}$ which does not contain $5$ elements in arithmetic progression. Then there exists a constant $c\in (0,1)$ such that
\[r_5(N)\ll \frac{N}{\exp((\log\log N)^{c})}.\]
Our work is a consequence of recent improved bounds on the $U^4$-inverse theorem of the first author and the fact that $3$-step nilsequences may be approximated by locally cubic functions on shifted Bohr sets. This combined with the density increment strategy of Heath-Brown and Szemer{\'e}di, codified by Green and Tao, gives the desired result.
\end{abstract}

\section{Introduction}\label{sec:introduction}

Let $[N] = \{1,\ldots, N\}$ and $r_k(N)$ denote the size of the largest $S\subseteq\{1,\ldots, N\}$ such that $S$ has no $k$-term arithmetic progressions. The first nontrivial upper bound on $r_3(N)$ came from work of Roth \cite{Rot54} which proved
\[r_3(N)\ll N(\log\log N)^{-1}.\]
A long series of works improved the bounds in the case $k=3$. This includes works of Heath-Brown \cite{Hea87}, Szemer{\'e}di \cite{Sze90}, Bourgain \cite{Bou99,Bou08}, Sanders \cite{San12, San11}, Bloom \cite{Blo16}, and Bloom and Sisask \cite{BS20}. Recently, in breakthrough work, Kelley and Meka \cite{KM23} proved that 
\[r_3(N)\ll N\exp(-\Omega((\log N)^{1/12}));\]
modulo the constant $1/12$ this matches the lower bound construction of Behrend \cite{Beh46}. The constant $1/12$ was refined to $1/9$ in work of Bloom and Sisask \cite{BS23}.

For higher $k$, establishing that $r_k(N) = o(N)$ was a long standing conjecture of Erd\H{o}s and Tur\'{a}n. In seminal works, Szemer\'{e}di \cite{Sze70} first established the estimate $r_4(N) = o(N)$
and then in later work Szemer\'{e}di \cite{Sze75} established his eponymous theorem that
\[r_k(N) = o(N).\]
Due to the uses of van der Waerden theorem and the regularity lemma (which was introduced in this work), Szemer\'{e}di's estimate was exceedingly weak. In seminal work Gowers \cite{Gow98,Gow01} introduced higher order Fourier analysis and proved the first ``reasonable'' bounds for Szemer\'{e}di's theorem:
\[r_k(N) \ll N(\log\log N)^{-c_k}.\]
The only previous improvement to this result for $k\ge 4$ was work on the case $k=4$ of Green and Tao \cite{GT09,GT17} which ultimately established that 
\[r_4(N)\ll N(\log N)^{-c}.\]

Our main result is a ``quasi-logarithmic'' bound in the case $k = 5$.
\begin{theorem}\label{thm:main}
There is $c\in(0,1)$ such that
\[r_5(N)\ll\frac{N}{\exp((\log\log N)^{c})}.\]
\end{theorem}
\begin{remark}
Throughout this paper, we will abusively write $\log$ for $\max(\log(\cdot), e^e)$. This is to avoid trivial issues regarding inputs which are small. 
\end{remark}

Our work relies crucially on a recent improved inverse theorem for the Gowers $U^4$-norm due to the first author \cite[Theorem~7]{Len23}. We first formally define the Gowers $U^k$-norm.
\begin{definition}\label{def:gowers-norm}
Given $f\colon\mb{Z}/N\mb{Z}\to\mb{C}$ and $k\ge 1$, we define
\[\snorm{f}_{U^k}^{2^k}=\mb{E}_{x,h_1,\ldots,h_k\in\mb{Z}/N\mb{Z}}\Delta_{h_1,\ldots,h_k}f(x)\]
where $\Delta_hf(x)=f(x)\ol{f(x+h)}$ is the multiplicative discrete derivative (extended to vectors $h$ in the natural way). This is known to be well-defined and a norm (seminorm if $k=1$).
\end{definition}
\begin{theorem}\label{thm:U4-new}
There exists an absolute constant $C\ge 1$ such that the following holds. Let $N$ be prime, let $\delta>0$, and suppose that $f\colon\mb{Z}/N\mb{Z}\to\mb{C}$ is a $1$-bounded function with 
\[\snorm{f}_{U^4(\mb{Z}/N\mb{Z})}\ge\delta.\]
There exists a degree $3$ nilsequence $F(g(n)\Gamma)$ such that it has dimension bounded by $(\log(1/\delta))^C$, complexity bounded by $C$, such that $F$ is $1$-Lipschitz, and such that
\[\big|\mb{E}_{n\in[N]} f(n)\ol{F(g(n)\Gamma)}\big|\ge1/\exp((\log(1/\delta))^{C}).\]
\end{theorem}

A key maneuver in this paper is our deduction of a variant of the $U^4$-inverse theorem which lends itself to the analysis of multiple nilsequences simultaneously and may be of independent interest. Although it is known that having a large $U^4$-norm does not necessarily imply direct correlation with a cubic phase function due to the existence of bracket polynomials, one can hope to achieve correlation on a dense, structured host set (for us, a Bohr set).
\begin{proposition}\label{prop:U4-inv-cubic}
There exists an absolute constant $C\ge 1$ such that the following holds. Let $N$ be prime, let $\delta>0$, and suppose that $f\colon\mb{Z}/N\mb{Z}\to\mb{C}$ is a $1$-bounded function with 
\[\snorm{f}_{U^4(\mb{Z}/N\mb{Z})}\ge\delta.\]
There exist $S\subseteq (1/N)\mb{Z}$ with $|S|\ll (\log(1/\delta))^C$, and $y\in\mb{Z}/N\mb{Z}$ such that the following holds: there is a locally cubic phase function $\phi\colon y+B(S,1/100)\to\mb{R}/\mb{Z}$ such that
\[|\mb{E}_{t\in\mb{Z}/N\mb{Z}}\mbm{1}_{t\in y + B(S,1/100)}f(t)e(-\phi(t))|\gg1/\exp((\log(1/\delta))^C).\]
\end{proposition}

We refer the reader to \cref{def:regular-bohr-set,def:local-degree-s} for precise definitions of Bohr set and locally cubic phase function. We remark that an analogous result to \cref{prop:3-step-conversion} for higher $U^k$-norms is false; this can be seen by examining the function $e(\{an\{bn\}\}\{cn\} dn)$. We discuss this issue in more detail in \cref{sec:outline}.

\subsection{Proof outline}\label{sec:outline}
The starting point of our work involves combining the density increment strategy of Heath-Brown \cite{Hea87} and Szemer{\'e}di \cite{Sze90}, which was reformulated in a robust manner by Green and Tao \cite{GT09} when proving that $r_4(N)\ll N\exp(-\Omega(\sqrt{\log\log N}))$, with the improved $U^4$-inverse theorem \cref{thm:U4-new} of the first author. The crucial difference between the density increment strategy of Roth \cite{Rot54} versus Heath-Brown \cite{Hea87} and Szemer{\'e}di \cite{Sze90} is that one finds correlations with ``many functions'' to deduce a density increment.

If we apply \cref{thm:U4-new} directly with the density increment strategy as codified by Green and Tao \cite{GT09}, we would at the very least need that, given a polynomial sequence $g(n)$ with $g(0) = \on{id}_{G}$ on a nilmanifold $G/\Gamma$ of degree $3$ with complexity $M$ and dimension $d$, we have 
\[\min_{1\le n\le N}d_{G/\Gamma}(\on{id}_G,g(n))\ll M^{O(d^{O(1)})}N^{-1/d^{O(1)}}.\]
When $g(n)$ is a polynomial sequence of degree $3$ on the torus such results can be derived from work of Schmidt \cite{Sch77} on small fractional parts of polynomials. While directly deriving such a result for nonabelian nilmanifolds does not appear implausible, at present the distribution theory of nilmanifolds only has such ``polynomial in $d$'' dependencies when dealing with test functions having vertical frequency, due to work of the first author \cite{Len23b}.

The crucial step in our work therefore is solving such a Schmidt-type problem for nilmanifolds via representing $3$-step nilmanifolds as sums of exponentials of locally cubic functions on Bohr sets (see \cref{prop:3-step-conversion}). The analogue of such a result for $2$-step nilmanifolds (without bounds) appears in work of Green and Tao \cite[Proposition~2.3]{GT08}. That such a result is true is a miracle of small numbers which is most easily seen via bracket polynomials. (In work of Green and Tao \cite{GT09} regarding $r_4(N)$, one operates with a version of the $U^3$-inverse theorem proven directly for locally quadratic functions on Bohr sets.)

At an informal level, the $U^4$-inverse theorem may be rephrased as follows: if $\snorm{f}_{U^4}$ is large for $1$-bounded $f$, then $f$ correlates with a bracket exponential phase $e(H(n))$ where $H(n)$ is (essentially) a sum of terms of the form 
\[an^3, an^2\{bn\}, an\{bn^2\}, an\{bn\}\{cn\}, an\{\{bn\}\{cn\}\}, \{an\}\{bn\}\{cn\}, an\{bn\{cn\}\} \imod 1\]
plus terms which are obviously of ``lower degree''. By the work of Green and Tao, various lower order terms may be viewed as quadratic functions on Bohr sets. Now note that 
\[\{x\}y + y\{x\} =  xy - \lfloor x\rfloor \lfloor y\rfloor  + \{x\}\{y\}.\]
Therefore 
\[\{x\}y + y\{x\} = xy + \{x\}\{y\} \imod 1.\]
Furthermore note that $e(\{x\}\{y\})$ is after appropriate smoothing a Lipschitz function on $(\mb{R}/\mb{Z})^2$ and therefore is well-approximated by a weighted sum of exponentials of the form $e(kx + \ell y)$ for $k,\ell\in \mb{Z}$. Given this (and noting the analogous fact for $e(\{x\}\{y\}\{z\})$) we may rewrite our basis of degree $3$ functions as 
\[an^3, an^2\{bn\}, an\{bn\}\{cn\} \imod 1;\]
the most crucial of these manipulations is 
\[an\{bn\{cn\}\} = abn^2\{cn\} - \{an\}bn\{cn\} \imod 1.\]

The miracle is that we do not have any nested $\{\cdot\}$ and all of the brackets surround linear functions. Therefore, upon restricting the fractional parts to lie in certain narrow ranges away from the discontinuities in the fraction part (i.e., Bohr set-type conditions) the functions $an^2\{bn\}, an\{bn\}\{cn\}$ are seen to be ``locally cubic''. That is, discrete fourth-order derivatives vanish given that all points in the corresponding $4$-dimensional hypercube lie in an appropriate Bohr set. The existence of such a miracle can be seen by examining carefully all ``fractional part'' expressions required in work of Green--Tao--Ziegler \cite[Appendix~E]{GTZ11} on the $U^4$-inverse theorem. For the $U^5$-inverse theorem and higher, we have the function $e(\{an\{bn\}\}\{cn\} dn)$ and fractional part identities do not allow one to ``remove iterated floor functions''.

To actually prove the desired representation of a step $3$ nilsequence, we partition the nilmanifold via a partition of unity. Operating within a partition of unity, we may manipulate the nilmanifold (as in \cite[Proposition~2.3]{GT08}) via explicitly operating in Mal'cev coordinates of the first kind. This allows us to manipulate the floor and fractional expressions as suggested by the bracket polynomial formalism in the previous paragraphs. Identifying various fundamental domains with the torus and applying Fourier expansion appropriately eventually gives the desired decomposition.

We remark for technical reasons it turns out to be useful to manipulate our nilsequence to be $N$-periodic and living on a nilmanifold for which the Lie bracket structure constants are integers which are sufficiently divisible. The first task is accomplished via appropriately quantifying a construction of Manners \cite{Man14} (see \cref{prop:phi-version}) and the second is accomplished via a ``clearing denominators'' trick on the Mal'cev basis (see \cref{lem:nil-integral}). One can see that such a manipulation might be useful by noting that if the structure constants of $G/\Gamma$ are sufficiently divisible then $\psi_{\exp}(\Gamma) = \mb{Z}^d$ where $\psi_{\exp}$ is the map to Mal'cev coordinates of the first kind (see \cref{lem:structure-constant}). In general $\psi_{\exp}(\Gamma)$ is not even a lattice in $\mb{Q}^d$.

Finally, given a correlation with a locally cubic function on a Bohr set, we are now in position to run the density increment strategy of Heath-Brown \cite{Hea87} and Szemer{\'e}di \cite{Sze90} as codified by Green and Tao \cite{GT09}. Our proof is very close to that of Green and Tao \cite{GT09}, although there are slight simplifications afforded in our situation since the dimension of our Bohr sets are ``quasi-logarithmic''. In particular, it is possible to operate by considering single atoms in the Bohr partition and avoid the use of regular Bohr sets. Our situation at this point is closely analogous to having access to the $U^3$-inverse theorem of Sanders \cite{San12b} and aiming to prove bounds of the form $r_4(N)\ll N \exp(-(\log\log N)^c)$.

\subsection*{Notation}
We use standard asymptotic notation. Given functions $f=f(n)$ and $g=g(n)$, we write $f=O(g)$, $f \ll g$, $g=\Omega(f)$, or $g\gg f$ to mean that there is a constant $C$ such that $|f(n)|\le Cg(n)$ for sufficiently large $n$. We write $f\asymp g$ or $f=\Theta(g)$ to mean that $f\ll g$ and $g\ll f$, and write $f=o(g)$ or $g=\omega(f)$ to mean $f(n)/g(n)\to0$ as $n\to\infty$. Subscripts on asymptotic notation indicate dependence of the bounds on those parameters. We will use the notation $[x] = \{1,2\ldots,\lfloor x\rfloor\}$.

We use the notations of \cref{sec:nilmanifolds} with regards to nilmanifolds. We write $\Delta_hf(x)=f(x)\ol{f(x+h)}$ for the multiplicative discete derivative and $\partial_hf(x)=f(x)-f(x+h)$ for the additive discrete derivative (for functions over appropriate domains and codomains). The Gowers $U^s$-norm on a finite abelian group $G$ is then defined via
\[\snorm{f}_{U^s}^{2^s}:=\mb{E}_{x,h_1,\ldots,h_s\in G}\Delta_{h_1,\ldots,h_s}f(x)\]
for $f\colon G\to\mb{C}$, which is known to be well-defined and a norm for $s\ge 2$ and a seminorm for $s=1$.

\subsection{Organization of paper}
In \cref{sec:local-cubic} we prove the main technical result regarding approximating nilsequences as local cubics on Bohr sets. In \cref{sec:deduc} we deduce the main result via a density increment argument. In \cref{sec:nilmanifolds} we collect various conventions and definitions regarding nilmanifolds. In \cref{sec:schmidt} we essentially reproduce \cite[Appendix~A]{GT09} and note that it verbatim extends to higher degree polynomials. In \cref{sec:periodic}, we manipulate \cref{thm:U4-new} to prove \cref{thm:U4-modified} which gives correlation with a periodic nilsequence where the underlying nilmanifold has appropriately divisible Lie bracket structure constants. Finally, in \cref{sec:deferred} we collect a number of deferred technical lemmas. 

\subsection*{Acknowledgements}
We thank Zach Hunter for various minor corrections. 

\section{Converting nilsequences to local cubics on a Bohr set}\label{sec:local-cubic}
The key idea is to work with a presentation of our nilsequence coming from \cref{thm:U4-new}, and manipulate it into an approximate form composed of locally cubic functions on Bohr sets. 

We will first define Bohr sets formally.
\begin{definition}\label{def:regular-bohr-set}
Given $S\subseteq\mb{Z}/N\mb{Z}$ and $\rho\in(0,1)$, we define the (centered) Bohr set
\[B(S,\rho):=\{x\in\mb{Z}/N\mb{Z}\colon\snorm{\xi x/N}_{\mb{R}/\mb{Z}}<\rho\text{ for all }\xi\in S\}.\]
Given $\alpha=(\alpha_\xi)_{\xi\in S}\in(\mb{R}/\mb{Z})^S$, we define the uncentered Bohr set
\[B_\alpha(S,\rho):=\{x\in\mb{Z}/N\mb{Z}\colon\snorm{\xi x/N-a_\xi}_{\mb{R}/\mb{Z}}<\rho\text{ for all }\xi\in S\}.\]
The parameter $|S|$ is the rank and $\rho$ is the radius.
\end{definition}

We next define the notion of local polynomial structure; we will care specifically about the case of degree $s=3$.
\begin{definition}\label{def:local-degree-s}
Given $S\subseteq G$ and $f\colon S\to H$, we say that $f$ is locally degree $s\ge0$ if for all $x,h_1,\ldots,h_{s+1}$ such that $x+\sum_{j=1}^{s+1}\epsilon_jh_j\in S$ for all $\epsilon=(\epsilon_1,\ldots,\epsilon_{s+1})\in\{0,1\}^{s+1}$, we have
\[\partial_{h_1,\ldots,h_{s+1}}f(x)=0.\]
\end{definition}
\begin{remark}
We will primarily be concerned with $S\subseteq\mb{Z}/N\mb{Z}$ and $H = \mb{R}/\mb{Z}$.
\end{remark}

\begin{proposition}\label{prop:3-step-conversion}
Suppose we are given a degree $3$ nilmanifold $G/\Gamma$ of dimension $d$, complexity $M$, and all Lie bracket structure constants divisible by $12$. Furthermore suppose that $F\colon G/\Gamma\to\mb{C}$ is smooth with Lipschitz norm bounded by $1$ and let $\eta\in(0,1/2)$. 

Then there exist data such that 
\[\sup_{n\in\mb{Z}/N\mb{Z}}\bigg|F(g(n)\Gamma) - \sum_{i\in I}w_i\mbm{1}_{n\in y_i + B(S,1/100)}e(\phi_i(n))\bigg|\le \eta\] 
such that:
\begin{itemize}
    \item $S\subseteq (1/N)\mb{Z}$ with size at most $d+1$;
    \item $I$ is an index set of size at most $O(M \eta^{-1})^{O(d^{O(1)})}$ and $|w_i|\le O(M \eta^{-1})^{O(d^{O(1)})}$ for all $i\in I$;
    \item $y_i\in\mb{Z}/N\mb{Z}$ for all $i\in I$;
    \item $\phi_i\colon\mb{Z}/N\mb{Z}\to\mb{R}/\mb{Z}$ is locally cubic on $y_i + B(S,1/100)$.
\end{itemize}
\end{proposition}

Then, \cref{prop:U4-inv-cubic} directly follows from a slightly modified version of \cref{thm:U4-new} (namely \cref{thm:U4-modified}) and \cref{prop:3-step-conversion}. \cref{thm:U4-modified} allows one to essentially assume that the underlying nilsequence is $N$-periodic and various structure constants are sufficiently divisible.

\begin{proof}[Proof of \cref{prop:U4-inv-cubic}]
By \cref{thm:U4-modified}, there exists a nilmanifold $G/\Gamma$ with function $F$ and a polynomial sequence $g(n)$ such that
\[\big|\mb{E}_{n\in \mb{Z}/N\mb{Z}}f(n)\ol{F(g(n)\Gamma)}\big|\ge 1/\exp((\log(1/\delta))^C) =: 2\eta\]
with $g(0) = \on{id}_G$ and $G/\Gamma$ having dimension $d$, complexity at most $M$ and all structure constants divisible by $12$ with 
\begin{align*}
d&\le (\log(1/\delta))^{O(1)}\\
M&\le \exp((\log(1/\delta))^{O(1)}).
\end{align*}

We now approximate $F(g(n)\Gamma)$ by a function $H(n)$ such that $\sup_{n\in \mb{Z}/N\mb{Z}}|F(g(n)\Gamma)-H(n)|\le\eta$ using \cref{prop:3-step-conversion}. As $f$ is $1$-bounded, we immediately have 
\[\big|\mb{E}_{n\in\mb{Z}/N\mb{Z}}f(n)\ol{H(n)}\big|\ge \big|\mb{E}_{n\in\mb{Z}/N\mb{Z}}f(n)\ol{F(g(n)\Gamma)}\big| - \mb{E}_{n\in\mb{Z}/N\mb{Z}}|H(n)-F(g(n)\Gamma)|\ge \eta.\]
We have some additional guarantees on the structure of $H$ (in particular, some associated data $S,I,w_i,y_i,\phi_i$). Applying Pigeonhole on the index set $I$ and using that the $w_i$ are sufficiently bounded, there exist $y\in\mb{Z}/N\mb{Z}$, $S$ a set of at most size $d+1$, $\rho=1/100$, and $\phi\colon\mb{Z}/N\mb{Z}\to\mb{R}/\mb{Z}$ locally cubic on $y + B(S,\rho)$ such that 
\[\big|\mb{E}_{n\in\mb{Z}/N\mb{Z}}f(n) \mbm{1}_{n\in y+ B(S,\rho)}e(-\phi(n))\big|\ge 1/\exp((\log(1/\delta))^{O(1)}). \qedhere\]
\end{proof}

In order to prove \cref{prop:3-step-conversion}, we will need a number of quantitative and structural lemmas about nilsequences. These arguments are deferred to \cref{sec:deferred}. We first require the following quantitative partition of unity for nilmanifolds.

\begin{lemma}\label{lem:nilmanifold-partition-of-unity}
Fix $\eps\in (0,1/2)$ and a nilmanifold $G/\Gamma$ of degree $s$, dimension $d$, and complexity $M$. There exists an index set $I$ and a collection of nonnegative smooth functions $\tau_j\colon G/\Gamma\to\mb{R}$ for $j\in I$ such that:
\begin{itemize}
    \item $L:= (M/\eps)^{O_s(d^{O_s(1)})}$;
    \item For all $g\in G$, we have $\sum_{j\in I}\tau_j(g\Gamma) = 1$;
    \item $|I|\le L$;
    \item For each $j\in I$, there exists $\beta\in [-L,L]^{d}$ so that for any $g\Gamma \in \on{supp}(\tau_j)$ there exist $g'\in g\Gamma$ such that $\log_G(g')\in \prod_{i=1}^d[\beta_i-\eps,\beta_i+\eps]$;
    \item $\tau_j$ are $L$-Lipschitz on $G/\Gamma$.
\end{itemize}
\end{lemma}

We also require the algebraic fact that if the Lie bracket structure constants of $G/\Gamma$ are sufficiently divisible then $\psi_{\exp}(\Gamma) = \mb{Z}^d$.
\begin{lemma}\label{lem:structure-constant}
There exists and integer $C_s\ge 1$ such that the following holds. Fix a nilmanifold $G/\Gamma$ of degree $s$ and a Mal'cev basis $\mc{X}$ such that all Lie bracket structure constants of $G/\Gamma$ are divisible by $C_s$. Then $\psi_{\exp}(\Gamma) = \mb{Z}^d$.
\end{lemma}
\begin{remark}
We may take $C_3 = 12$.
\end{remark}

Next we need the conversion between the nilmanifold distance and distance in coordinates of the first kind (on $\mb{R}^d$).
\begin{lemma}\label{lem:first-to-torus}
Fix a nilmanifold $G/\Gamma$ of degree $s$, dimension $d$, and complexity $M$. If $\snorm{\psi_{\exp}(x) - \psi_{\exp}(y)}_\infty\le\eps$ and $\snorm{\psi_{\exp}(x)}_\infty\le L$ then
\[d_{G/\Gamma}(x\Gamma,y\Gamma)\le\eps(LM)^{O_s(d^{O_s(1)})}.\]
\end{lemma}

We also require the following basic estimate regarding Fourier expansion of Lipschitz functions on the torus; this follows immediately by quantifying the proof in \cite[Proposition~1.1.13]{Tao12} (see e.g.~\cite[Lemma~A.8]{PSS23}).
\begin{lemma}\label{lem:torus-expand}
Fix $0<\eps<1/2$, and let $F\colon(\mb{R}/\mb{Z})^d\to\mb{C}$ with $\snorm{F}_{\mr{Lip}}\le L$ with metric $d(x,y) = \max_{1\le i\le d}\snorm{x_i-y_i}_{\mb{R}/\mb{Z}}$ for $x,y\in (\mb{R}/\mb{Z})^d$. There exists an absolute constant $C>0$ such that there exist $c_{\xi}$ with $\sum_{\xi}|c_{\xi}|\le (3CLd\eps^{-1})^{5d}$ and
\[\sup_{x\in (\mb{R}/\mb{Z})^d}\bigg|F(x) -  \sum_{|\xi|\le (CLd\eps^{-1})^{2}}c_{\xi}e(\xi\cdot x)\bigg|\le \eps.\]
\end{lemma}

We finally require the following elementary lemma regarding how local degree acts with respect to multiplication; this is once again deferred to \cref{sec:deferred}. (Recall we are defining local degree with respect to $\partial$ as opposed to $\Delta$, and the following cannot be altered to only require that the functions reduced $\imod{1}$ have local degree.)

\begin{lemma}\label{lem:product-degree}
If $S\subseteq G$ and $f_j\colon S\to\mb{C}$ are locally degree $d_j$ functions, then $g=\prod_{j=1}^kf_k$ is a locally degree $\sum_{j=1}^kd_j$ function.
\end{lemma}

We are ready to proceed.
\begin{proof}[Proof of \cref{prop:3-step-conversion}]
We break the proof into a series of steps. 

\noindent\textbf{Step 1: Polynomial sequence and group multiplication in coordinates.}
Note that as we are dealing with nilpotent groups of step at most $3$ we have
\[\log(e^Xe^Y)=X+Y+\frac{1}{2}[X,Y]+\frac{1}{12}([X,[X,Y]]-[Y,[X,Y]])\]
by the Baker--Campbell--Hausdorff formula. Let $d_j=\dim G_{1}-\dim G_{j+1}$ and let our Mal'cev basis be $(X_k)_{k\in[d]}$ such that $(X_k)_{k>d_j}$ spans $\log G_{j+1}$ for $j\in\{0,1,2\}$. Note that by definition $\Gamma$ is the set of elements of the form $\prod_{i=1}^{d}\exp(t_iX_i)$ for $t_i\in\mb{Z}$. Recall that we have
\[[X_i,X_j]=\sum_{k=1}^dc_{ijk}X_k\]
for $c_{ijk}$ integers divisible by $12$ and of absolute value at most $M$ due to the complexity assumption. As $[\log G_i, \log G_j]\subseteq \log G_{i+j}$, we have that $c_{ijk}=0$ for $k\le d_{a+b+1}$ if $i>d_a$ and $j>d_b$.

As $g(n)$ is a polynomial sequence with $g(0) = \on{id}_G$, by Taylor expansion we have 
\[g(n) = g_1^{n}g_2^{\binom{n}{2}}g_3^{\binom{n}{3}}\]
with $g_i\in G_i$. Therefore
\begin{align*}
\log(g(n)) &= \log\bigg(\exp(n\log(g_1))\exp\bigg(\binom{n}{2}\log(g_2)\bigg)\exp\bigg(\binom{n}{3}\log(g_3)\bigg)\bigg)\\
&=n\log(g_1) + \binom{n}{2}\log(g_2) + \frac{n}{2}\binom{n}{2}[\log(g_1),\log(g_2)] + \binom{n}{3}\log(g_3)
\end{align*}
with $[\log(g_1),\log(g_2)]\in\log(G_3)$. Therefore we have 
\begin{equation}\label{eq:nilsequence-first-kind}
g(n)=\exp\bigg(\sum_{j=1}^{d}p_j(n)X_j\bigg)
\end{equation}
with $p_j(0) = 0$ and $p_j$ being at most degree $1$ if $1\le j\le d_1$, at most degree $2$ if $d_1+1\le j\le d_2$, and at most degree $3$ if $d_2+1\le j\le d_3=d$. We next realize multiplication of two elements corresponding to Mal'cev coordinates of the first kind as
\begin{align*}
&(x_1,\ldots,x_d)\ast(y_1,\ldots,y_d)\\
&=(x_1+y_1,\ldots,x_{d_1}+y_{d_1},x_{d_1+1}+y_{d_1+1}+\phi_{d_1+1}(x_{\le d_1},y_{\le d_1}),\ldots,x_{d_2}+y_{d_2}+\phi_{d_2}(x_{\le d_1},y_{\le d_1}),\\
&x_{d_2+1}+y_{d_2+1}+\phi_{d_2+1}(x_{\le d_2},y_{\le d_2})+\varphi_{d_2+1}(x_{\le d_1},y_{\le d_1}),\ldots,x_d+y_d+\phi_d(x_{\le d_2},y_{\le d_2})+\varphi_d(x_{\le d_1},y_{\le d_1}))
\end{align*}
where $\phi_j$ are bilinear forms with \emph{integral} coefficients and $\varphi_j$ are cubic forms with \emph{integral} coefficients. This is an immediate consequence of the Baker--Campbell--Hausdorff formula and using the assumption that the Lie bracket structure constants are all divisible by $12$. Furthermore these coefficients are of height at most $O(Md)^{O(1)}$ and since $[G_2,G_2]\subseteq G_4=\on{Id}_G$, the bilinear form $\phi_j$ for $j\in[d_1+1,d_2]$ has all coordinates $0$ on the box $[d_1+1,d_2]^2$.

\noindent\textbf{Step 2: Explicit representation of nilsequence with coordinates.}
We can represent the coordinates of $g(n)\Gamma$ in a fundamental domain with respect to Mal'cev coordinates of the first kind via iterated floor and fractional parts. These coordinates are not smooth at the boundary and thus we decompose $F$ via a partition of unity (and using different fundamental domains for each part). This will allow us to manipulate such coordinate functions without needing to worry very precisely about the minor discontinuities.

Let $L = O(M)^{O(d^{O(1)})}$. We write
\[F=\sum_{i\in I}F_i\]
with $F_i=F\tau_i$ where $\tau_i$ is as in \cref{lem:nilmanifold-partition-of-unity} applied with parameter $\eps=10^{-3}$. This will allow us to represent $F\tau_i$ on the fundamental domain (with respect to Mal'cev coordinates of the first kind) $\prod_{j=1}^d[\beta_j^{(i)}-1/2,\beta_j^{(i)}+1/2)$, where $\beta^{(i)}\in [-L,L]^{d}$ and
\[\on{supp}(F\tau_i)\subseteq\on{supp}(\tau_i)\subseteq\prod_{j=1}^d[\beta_j^{(i)}-10^{-3},\beta_j^{(i)}+10^{-3}).\]

We now define nonstandard (shifted) floor and fractional part functions for each $i\in I$ and $j\in[d]$ so that 
\begin{align*}
\{t\}_{i,j} &\equiv t \imod 1\\
t &= \{t\}_{i,j} + \lfloor t\rfloor_{i,j}\\
\{t\}_{i,j}&\in [\beta_j^{(i)}-1/2,\beta_j^{(i)}+1/2).
\end{align*}

For nearly the entire remainder of the proof, we will focus on massaging the representation of $F_i(x)$ into a more convenient form. Given $x\in G$ such that $\psi_{\exp}(x)= (x_1,\ldots,x_d)$, consider $\gamma$ with Mal'cev coordinates of the first kind:
\begin{align}
&(-\lfloor x_j\rfloor_{i,j})_{j\in[d_1]},\quad(-\lfloor x_j+\phi_j(x_{\le d_1},-\lfloor x_{\le d_1}\rfloor_i)\rfloor_{i,j})_{d_1+1\le j\le d_2},\notag\\
&(-\lfloor x_j+\varphi_j(x_{\le d_1},-\lfloor x_{\le d_1}\rfloor_i)+\phi_j(x_{\le d_2},x_{\le d_2}^\ast)\rfloor_{i,j})_{d_2+1\le j\le d_3},\label{eq:fundamental-shift}
\end{align}
where $x_{\le d_2}^\ast$ has coordinates equal to those on the first line of \cref{eq:fundamental-shift} (so it implicitly depends on $i$) and $\lfloor x_{\le d_1}\rfloor_i = (\lfloor x_1\rfloor_{i,1},\ldots, \lfloor x_1\rfloor_{i,d_1})$. Furthermore let $\{x_{\le d_1}\}_i = x - \lfloor x_{\le d_1}\rfloor_i$. Note that $\gamma\in \Gamma$ by \cref{lem:structure-constant}.

Thus $x\gamma$ has coordinates 
\begin{align}
&(\{x_j\}_{i,j})_{j\in[d_1]},\quad(\{x_j+\phi_j(x_{\le d_1},-\lfloor x_{\le d_1}\rfloor_i)\}_{i,j})_{d_1+1\le j\le d_2},\notag\\
&(\{x_j+\varphi_j(x_{\le d_1},-\lfloor x_{\le d_1}\rfloor_i)+\phi_j(x_{\le d_2},x_{\le d_2}^\ast)\}_{i,j})_{d_2+1\le j\le d}\label{eq:fundamental-coordinates-1}
\end{align}
which is in the specified fundamental domain $\prod_{j=1}^d[\beta_j^{(i)}-1/2,\beta_j^{(i)}+1/2)$ for $F_i$. Recall also that $\phi_j$ for $j\in[d_2+1,d]$ has certain $0$ coefficients. 

Let $x^{\ast}_{[d_1+1,d_2]}$ denote the vector with $d_2$ coordinates, the first $d_1$ of which are zero and the last $d_2-d_1$ of which are $(-\lfloor x_j+\phi_j(x_{\le d_1},-\lfloor x_{\le d_1}\rfloor_i)\rfloor_{i,j})_{d_1+1\le j\le d_2}$. Let $y_j = x_j+\phi_j(x_{\le d_1},-\lfloor x_{\le d_1}\rfloor_i)$, $y^{\ast}_{[d_1+1,d_2]}$ analogously be a vector of these coordinates and $d_1$ many $0$s, and let $\{y_{[d_1+1,d_2]}\}_i^\ast = (0,\ldots,0,\{y_{d_1+1}\}_{i,d_1+1},\ldots,\{y_{d_2}\}_{i,d_2})$.

Note that we have the identity
\[-y\lfloor z\rfloor_2=\lfloor y\rfloor_1 z-yz-\lfloor y\rfloor_1\lfloor z\rfloor_2+\{y\}_1\{z\}_2,\]
where the subscripts $1$ and $2$ indicate potentially different shift types. Therefore 
\begin{align*}
\phi_j(x_{\le d_2},x_{\le d_2}^\ast) &= \phi_j(x_{\le d_2},-\lfloor x_{\le d_1}\rfloor_i) + \phi_j(x_{\le d_2},x^{\ast}_{[d_1+1,d_2]})\\
&= \phi_j(x_{\le d_2},-\lfloor x_{\le d_1}\rfloor_i) + \phi_j(x_{\le d_1},x^{\ast}_{[d_1+1,d_2]})\\
&= \phi_j(x_{\le d_2},-\lfloor x_{\le d_1}\rfloor_i) + \phi_j(\lfloor x_{\le d_1}\rfloor_i,y^{\ast}_{[d_1+1,d_2]})-\phi_j(x_{\le d_1},y^{\ast}_{[d_1+1,d_2]}) \\
&\qquad\qquad+ \phi_j(\lfloor x_{\le d_1}\rfloor_i, x^{\ast}_{[d_1+1,d_2]})+ \phi_j(\{x_{\le d_1}\}_i,\{y_{[d_1+1,d_2]}\}_i^\ast).
\end{align*}
The equality $\phi_j(x_{\le d_2},x^{\ast}_{[d_1+1,d_2]}) = \phi_j(x_{\le d_1},x^{\ast}_{[d_1+1,d_2]})$ follows as $\phi_j$ has no nonzero coordinates on the box $[d_1+1,d_2]^2$.

This implies that
\begin{align*}
&(\{x_j+\varphi_j(x_{\le d_1},-\lfloor x_{\le d_1}\rfloor_i)+\phi_j(x_{\le d_2},x_{\le d_2}^\ast)\}_{i,j})_{d_2+1\le j\le d}\\
&= (\{x_j+\varphi_j(x_{\le d_1},-\lfloor x_{\le d_1}\rfloor_i)+\phi_j(x_{\le d_2},-\lfloor x_{\le d_1}\rfloor_i) + \phi_j(\lfloor x_{\le d_1}\rfloor_i,y^{\ast}_{[d_1+1,d_2]})\\
&\qquad-\phi_j(x_{\le d_1},y^{\ast}_{[d_1+1,d_2]})+ \phi_j(\{x_{\le d_1}\}_i,\{y_{[d_1+1,d_2]}\}_i^\ast)\}_{i,j})_{d_2+1\le j\le d};
\end{align*}
we are able to drop $\phi_j(\lfloor x_{\le d_1}\rfloor_i, x^{\ast}_{[d_1+1,d_2]})$ as $\phi_j$ has integral coefficients and we are taking fractional parts. Using the above general identity and this expression, we thus have coordinates of the form 
\begin{align}
&(\{x_j\}_{i,j})_{j\in[d_1]},\quad(\{x_j-\phi_j(x_{\le d_1},\lfloor x_{\le d_1}\rfloor_i)\}_{i,j})_{d_1+1\le j\le d_2},\notag\\
&(\{x_j+\varphi_j^\ast(x_{\le d_1},\lfloor x_{\le d_1}\rfloor_i)+\phi_j^\ast(x_{\le d_2},\lfloor x_{\le d_1}\rfloor_i)+\sigma_j(x_{\le d_2},x_{\le d_1})+\phi_j(\{x_{\le d_1}\}_i,\{y_{[d_1+1,d_2]}\}^{\ast})\}_{i,j})_{d_2+1\le j\le d_3}.\label{eq:fundamental-coordinates-2}
\end{align}
Here $\varphi_j^{\ast}$ are degree at most $3$ polynomials and $\phi_j^{\ast},\sigma_j$ are bilinear forms. Additionally, all coefficients are integral with heights bounded by $(O(Md))^{O(1)}$.

Let us briefly take stock of what has been accomplished in the last manipulation. Note that there are no longer any ``iterated'' floor expressions and the only terms being ``floored'' are $x_{\le d_1}$. The Bohr sets we will ultimately take therefore are determined by these coordinates only. 

We now ``lift'' $F_i$ from a function on $G/\Gamma$ to $\wt{F}_i$ on $(\mb{R}/\mb{Z})^d$. This can be done via identifying $(\mb{R}/\mb{Z})^d$ with the fundamental domain $\prod_{j=1}^d[\beta_j^{(i)}-1/2,\beta_j^{(i)}+1/2)$ with respect to $\psi_{\exp}$. Now, $\wt{F}_i$ is seen to be $O(LM/\eps)^{O(d^{O(1)})}$-Lipschitz on $(\mb{R}/\mb{Z})^d$ by \cref{lem:first-to-torus}. (We are also using that the support of $\wt{F}_i$ is close to the center of the torus, so the torus metric and $\ell^\infty$ on $\mb{R}^d$ are the same where it matters.)

Therefore it is sufficient to consider $\wt{F}_i$ with coordinates given by 
\begin{align}
&(x_j)_{j\in[d_1]},\quad(x_j-\phi_j(x_{\le d_1},\lfloor x_{\le d_1}\rfloor_i))_{d_1+1\le j\le d_2},\notag\\
&(x_j+\varphi_j^\ast(x_{\le d_1},\lfloor x_{\le d_1}\rfloor_i)+\phi_j^\ast(x_{\le d_2},\lfloor x_{\le d_1}\rfloor_i)+\sigma_j(x_{\le d_2},x_{\le d_1})+\phi_j(\{x_{\le d_1}\}_i,\{y_{[d_1+1,d_2]}\}_i^\ast))_{d_2+1\le j\le d_3}\label{eq:fundamental-coordinates-3}
\end{align}
which live on the torus.

We now ``smooth'' $\{x_{\le d_1}\}_i$ and $\{y_{[d_1+1,d_2]}\}_i^\ast$. We replace $\{\cdot\}_{i,j}$ with a $1$-periodic function $\{\cdot\}_{i,j,\mr{sm}}$ which agrees with $\{x\}_{i,j}$ if $\{x\}_{i,j}\in [\beta_j^{(i)}-1/4,\beta_j^{(i)}+1/4)$ and is $O(1)$-Lipschitz when viewed as a function on $\mb{R}/\mb{Z}$. Given this we write $\{x_{\le d_1}\}_{\mr{sm}}$ and $\{y_{[d_1+1,d_2]}\}_{\mr{sm}}^\ast$ for the associated vectors where smooth fractional parts have been used. 

We claim that it suffices to consider $\wt{F}_i$ with coordinates given by 
\begin{align}
&(x_j)_{j\in[d_1]},\quad(x_j-\phi_j(x_{\le d_1},\lfloor x_{\le d_1}\rfloor_i))_{d_1+1\le j\le d_2},\notag\\
&(x_j+\varphi_j^\ast(x_{\le d_1},\lfloor x_{\le d_1}\rfloor_i)+\phi_j^\ast(x_{\le d_2},\lfloor x_{\le d_1}\rfloor_i)+\sigma_j(x_{\le d_2},x_{\le d_1})+\phi_j(\{x_{\le d_1}\}_{\mr{sm}},\{y_{[d_1+1,d_2]}\}_{\mr{sm}}^\ast))_{d_2+1\le j\le d_3}.\label{eq:fundamental-coordinates-4}
\end{align}
Note that if $\{x_{\le d_1}\}_{\mr{sm}}\neq\{x_{\le d_1}\}_i$ or $\{y_{[d_1+1,d_2]}\}_{\mr{sm}}\neq\{y_{[d_1+1,d_2]}\}_i^\ast$ we immediately see that one of the coordinates in the first two lines has been forced outside the support of $\wt{F}_i$ and thus the value is already by construction zero in both coordinates. Otherwise the coordinates in \cref{eq:fundamental-coordinates-3,eq:fundamental-coordinates-4} match and the representation has been unchanged. 

\noindent\textbf{Step 3: Fourier expansion of nilsequence.}
Now, $\wt{F}_i\colon(\mb{R}/\mb{Z})^d\to\mb{C}$ is a Lipschitz function. Therefore by \cref{lem:torus-expand} with parameter $\tau$,
\[\sup_{z\in (\mb{R}/\mb{Z})^d}\bigg|\wt{F}_i(z) - \sum_{|\xi|\le O(M\tau^{-1})^{O(d^{O(1)})}}c_{\xi} e(\xi\cdot z)\bigg|\le \tau\]
with $|c_{\xi}|\le O(M\tau^{-1})^{O(d^{O(1)})}$.

Using this Fourier representation, we have 
\begin{align*}
&\sup_{x\in G}\bigg|F_i(x\Gamma) - \sum_{|\xi|\le O(M\tau^{-1})^{O(d^{O(1)})}}c_{\xi} e\bigg(\sum_{j=1}^{d}T_{\xi,j}x_j + \wt{\varphi}_{\xi}(x_{\le d_1},\lfloor x_{\le d_1}\rfloor_i)+ \wt{\phi}_{\xi}(x_{\le d_2},\lfloor x_{\le d_1}\rfloor_i) \\
&\qquad\qquad\qquad\qquad\qquad\qquad\qquad\qquad\qquad+\wt{\sigma}_\xi(x_{\le d_1},x_{\le d_2})+ \wt{\phi}'_{\xi}(\{x_{\le d_1}\}_{\mr{sm}},\{y_{[d_1+1,d_2]}\}_{\mr{sm}}^\ast) \bigg) \bigg|\le \tau
\end{align*}
such that $\wt{\varphi}_\xi$ is at most a degree $3$ polynomial, and $\wt{\phi}_\xi,\wt{\sigma}_\xi,\wt{\phi}'_\xi$ are bilinear forms. Furthermore all coefficients involved are integers bounded by $(M\tau^{-1})^{O(d^{O(1)})}$. 

The final smoothing we perform before we specialize to the polynomial sequence under consideration is to remove the $\{\cdot\}_{\mr{sm}}\{\cdot\}_{\mr{sm}}$ terms. Note that the function
\[(x,y)\to e(T\{x\}_{i,j,\mr{sm}}\{y\}_{i,j,\mr{sm}})\]
is $O(T L^{O(1)})$-Lipschitz. We apply \cref{lem:torus-expand} with $\eps = (M \tau^{-1})^{-O(d^{O(1)})}$ sufficiently small and replace each of the possible $d_1\cdot (d_2-d_1)$ possible combinations for the ``smoothed parts'' simultaneously in each term. We find that
\begin{equation}\label{eq:big-fourier-expansion}
\sup_{x\in G}\bigg|F_i(x\Gamma) - \sum_{k\in I}c_k e\bigg(\sum_{j=1}^{d}T_{k,j}x_j + \wt{\varphi}_k(x_{\le d_1},\lfloor x_{\le d_1}\rfloor_i) + \wt{\phi}_k(x_{\le d_2},\lfloor x_{\le d_1}\rfloor_i)+\wt{\sigma}_k(x_{\le d_2},x_{\le d_1})\bigg) \bigg|\le 2\tau
\end{equation}
with $|I|\le (M \tau^{-1})^{O(d^{O(1)})}$, $|c_k|\le (M\tau^{-1})^{O(d^{O(1)})}$, $\wt{\varphi}_k$ is an at most degree $3$ polynomial, $\wt{\phi}_k$ and $\wt{\sigma}_k$ are bilinear forms, and all coefficients of these forms bounded by $(M \tau^{-1})^{O(d^{O(1)})}$. Explicitly, we used \cref{lem:torus-expand} on $(\mb{R}/\mb{Z})^2$ with parameter $\eps$ many times and multiplied the representations together.

We now specialize to the case of interest. Note that our primary concern is with the case where $x = g(n)$ and therefore $x_j = p_j(n)$ where $\deg p_j\le i$ for $j\le d_i$ if $i\in\{1,2,3\}$. Using the representation \cref{eq:big-fourier-expansion} and collecting terms we have
\begin{align*}
&\sup_{n\in \mb{Z}}\bigg|F_i(g(n)\Gamma) - \sum_{k\in I}c_k e(H_k(n)) \bigg|\le 2\tau
\end{align*}
where $H_k\colon\mb{Z}\to\mb{R}$ is a sum of terms of the form 
\begin{itemize}
    \item $\alpha n^3 + \beta n^2 + \gamma n$
    \item $(\alpha n^2 + \beta n + \gamma) \lfloor p_j(n)\rfloor_{i,j}$ for $1\le j\le d_1$
    \item $(\alpha n + \beta)\lfloor p_j(n)\rfloor_{i,j}\lfloor p_k(n)\rfloor_{i,k}$ for $1\le j\le k\le d_1$
    \item $\gamma\lfloor p_j(n)\rfloor_{i,j}\lfloor p_k(n)\rfloor_{i,k}\lfloor p_\ell(n)\rfloor_{i,\ell}$ for $1\le j\le k\le \ell \le d_1$;
\end{itemize}
note that one can always collect terms so that there are at most $d^{O(1)}$ terms in each expression. 

\noindent\textbf{Step 4: $N$-periodicity and introducing Bohr sets.} Let $\rho=1/100$, which will be the radius of our Bohr sets. Now, note that the expressions $e(H_k(n))$ are not ``obviously'' $N$-periodic. We will artificially make them so via a partition of unity argument. There exist smooth $\chi_1,\ldots,\chi_{10^{3}}$ such that $\chi_j\colon(\mb{R}/\mb{Z})\to \mb{R}$ are nonnegative with $\on{supp}(\chi_{j})\subseteq[j/10^{3}, (j+2)/10^{3}) \imod 1$ and $\sum_{j=1}^{10^{3}}\chi_j = 1$. We have 
\begin{align*}
&F_i(g(n)\Gamma) = \sum_{j\in [10^{3}]}\chi_j(n/N)F_i(g(n)\Gamma),\\
\sup_{\substack{n\in \mb{Z}\\ j\in [10^{3}]}}\bigg|&F_i(g(n)\Gamma)\chi_j(n/N) - \chi_j(n/N)\sum_{k\in I}c_k e(H_k(n)) \bigg|\le O(\tau).
\end{align*}

Now fix $j\in [10^{3}]$; such a term only contributes when $n/N \in [j/10^{3}, (j+2)/10^{3}) \mod 1$. Next note that each $p_\ell(n) = \alpha_\ell n$ for $\ell\le d_1$ and because the initial nilsequence $g(n)\Gamma$ is $N$-periodic, by \cite[Lemma~A.12]{Len23} we have that $\alpha_\ell \in (1/N)\mb{Z}$. 

Note that in order for $\chi_j(n/N)F_i(g(n)\Gamma)$ to be nonzero, we need that 
\begin{itemize}
    \item  $n/N \equiv [j/10^{3}, (j+2)/10^{3}) \mod 1$
    \item $\{p_\ell(n)\}_{i,\ell} \in [\beta_\ell^{(i)}-10^{-3},\beta_\ell^{(i)}+10^{-3})$ for all $\ell\in[d_1]$.
\end{itemize}
These conditions are clearly $N$-periodic and in fact all such $n$ lie in a shifted Bohr set with frequency set $S = \{1/N\} \cup \{\alpha_\ell\}_{1\le \ell\le d_1}\subseteq(1/N)\mb{Z}$. If the corresponding shifted Bohr set is empty we know that $\chi_j(n/N)F_i(g(n)\Gamma) = 0$ and we approximate $\chi_j(n/N)F_i(g(n)\Gamma)$ by $0$. Else there is $y_j\in\mb{Z}/N\mb{Z}$ such that $y_j/N \equiv [j/10^{3}, (j+2)/10^{3}) \mod 1$ and $\{\alpha_\ell y_j\}_{i,\ell} \in [\beta_\ell^{(i)}-10^{-3},\beta_\ell^{(i)}+10^{-3})$ for $\ell\in[d_1]$. Letting $J_i$ be the set of $j$ for which this shifted Bohr set is nonempty, we have 
\[\sup_{\substack{n\in \mb{Z}\\ j\in J_i}}\bigg|F_i(g(n)\Gamma)\chi_j(n/N) - \mbm{1}_{n\in y_j + B(S, \rho)}\chi_j(n/N)\sum_{k\in I}c_k e(H_k(n)) \bigg|\le O(\tau).\]

We now apply Fourier expansion on $\chi_j$ using \cref{lem:torus-expand} on the torus $\mb{R}/\mb{Z}$ and appropriately chosen error parameter $\tau' = (M\tau^{-1})^{-O(d^{O(1)})}$. We find
\begin{equation}\label{eq:penult-correlation-reduction}
\sup_{\substack{n\in\mb{Z}\\j\in J_i}}\bigg|F_i(g(n)\Gamma)\chi_j(n/N) - \mbm{1}_{n\in y_j + B(S,\rho)}\chi_j(n/N)\sum_{k\in I, |\xi|\le (\tau')^{-O(1)}}c_kd_{\xi} e(H_k(n) + \xi n/N) \bigg|\le O(\tau)
\end{equation}
with $|d_{\xi}|\le (\tau')^{-O(1)}$.

Given $n\in y_j + B(S,\rho)$, we will now replace $H_k(n)$ by $H_{k,j}(n)$ so that the latter is $N$-periodic (and it will be locally cubic on $y_j+B(S,\rho)$). We fix an interval $I_j \in [-N,N)$ of integers of length at most $\lceil N/50\rceil$ such that for each $n\in y_j + B(S,\rho)$ there is a unique integer $n'\in I_j$ such that $n'\equiv n \imod N$. This exists since $1/N\in S$. We then define $H_{k,j}(n) \equiv H_k(n') \mod 1$ for such $n\in y_j + B(S,\rho)$ and $H_{k,j}(n) = 0$ otherwise; note that $H_{k,j}(n)$ is taking values in $\mb{R}/\mb{Z}$ which is sufficient as we are plugging these values into the exponential function.

It is easy to see that \cref{eq:penult-correlation-reduction} holds with $H_k(n)+\xi\cdot n/N$ replaced by $H_{k,j}(n)$ since everything in the inequality except for potentially $H_k(n)$ is $N$-periodic (recall $S\subseteq(1/N)\mb{Z}$), and there is a supremum over $n\in\mb{Z}$ on the outside. So, we have
\[\sup_{\substack{n\in \mb{Z}\\ j\in J_i}}\bigg|F_i(g(n)\Gamma)\chi_j(n/N) - \mbm{1}_{n\in y_j + B(S, \rho)}\sum_{k\in I, |\xi|\le (\tau')^{-O(1)}}c_k d_{\xi} e(H_{k,j}(n)+\xi n/N) \bigg|\le O(\tau)\]
where $H_{k,j}(n)$ is $N$-periodic by construction. Taking $\tau = \eta M^{-O(d^{O(1)})}$ and summing over $j$ and $i$ then finishes the proof modulo showing that $H_{k,j}(n)$ is a locally cubic function on $y_j + B(S,\rho)$.

\noindent\textbf{Step 5: Local cubicity on shifted Bohr sets.}
Fix $i$, $j\in J_i$, and $k\in I$. We wish to show local cubicity of $H_{k,j}$ on $y_j+B(S,\rho)$. Suppose that $n,h_1,h_2,h_3\in\mb{Z}/N\mb{Z}$ are such that $n + \sum_{\ell=1}^{3}\epsilon_\ell h_\ell\in y_j + B(S,\rho)$ for all $(\epsilon_1,\epsilon_2,\epsilon_3)\in \{0,1\}^3$. We consider the representatives of $n + \sum_{\ell=1}^{3}\epsilon_\ell h_\ell$ in $I_j$; we see that there exist $n',h_1',h_2',h_3'\in \mb{Z}$ such that $n' + \sum_{\ell=1}^{3}\epsilon_\ell h_\ell' \equiv n + \sum_{\ell=1}^{3}\epsilon_\ell h_\ell \imod N$ and $n' + \sum_{\ell=1}^{3}\epsilon_\ell h_\ell'\in I_j$ for all $\epsilon\in\{0,1\}^3$. Indeed, one can look at the representatives of each $n + \sum_{\ell=1}^{3}\epsilon_\ell h_\ell$ in $I_j$, and use the fact that if $t_1 + t_2 \equiv t_3 + t_4 \imod N$ with $t_i\in I_j$ then $t_1 + t_2 = t_3 + t_4$ (since $I_j$ has length at most $\lceil N/50\rceil$).

Therefore it suffices to prove that $H_k\colon\mb{Z}\to\mb{R}$ is locally cubic on $y_j+B(S,\rho)$. Recalling the classification of $H_k(n)$ as a sum of terms of various kinds, it suffices to verify this for each individual function type which is combined to give $H_k(n)$.

Note that pure polynomials are always the correct degree (on all of $\mb{Z}$). By \cref{lem:product-degree} it suffices to verify that $\lfloor \alpha_\ell n \rfloor_{i,\ell}$ is locally linear on $y_j+B(S,\rho)$ for $1\le\ell\le d_1$.

Given $n,n + h_1, n+h_2, n+h_1+h_2\in y_j + B(S,\rho)$ we have that
\begin{align*}
&\lfloor \alpha_\ell n \rfloor_{i,\ell} - \lfloor \alpha_\ell (n+h_1) \rfloor_{i,\ell} - \lfloor \alpha_\ell (n+h_2) \rfloor_{i,\ell} + \lfloor \alpha_\ell (n+h_1+h_2) \rfloor_{i,\ell} \in \mb{Z},\\
&\big|\lfloor \alpha_\ell n \rfloor_{i,\ell} - \lfloor \alpha_\ell (n+h_1) \rfloor_{i,\ell} - \lfloor \alpha_\ell (n+h_2) \rfloor_{i,\ell} + \lfloor \alpha_\ell (n+h_1+h_2) \rfloor_{i,\ell} \big|\\
&= \big|\{ \alpha_\ell n \}_{i,\ell} - \{ \alpha_\ell (n+h_1) \}_{i,\ell} - \{ \alpha_\ell (n+h_2) \}_{i,\ell} + \{ \alpha_\ell (n+h_1+h_2) \}_{i,\ell} \big|\le 1/2,
\end{align*}
the inequality using that $\rho=1/100$ and every $\alpha_\ell(n+\epsilon_1h_1+\epsilon_2h_2)$ must be close in $\mb{R}/\mb{Z}$ to $\alpha_\ell y_j$. Additionally, we recall $\{\alpha_\ell y_j\}\in[\beta_\ell^{(i)}-10^{-3},\beta_\ell^{(i)}+10^{-3})$ which implies that there is no discontinuity in $\{\cdot\}_{i,\ell}$ in the area of consideration.

Therefore
\[\lfloor \alpha_\ell n \rfloor_{i,\ell} - \lfloor \alpha_\ell (n+h_1) \rfloor_{i,\ell} - \lfloor \alpha_\ell (n+h_2) \rfloor_{i,\ell} + \lfloor \alpha_\ell (n+h_1+h_2) \rfloor_{i,\ell} = 0\]
for $n,n + h_1, n+h_2, n+h_1+h_2\in y_j + B(S,\rho)$ as desired. This (finally) completes the proof.
\end{proof}

\section{Proof of \texorpdfstring{\cref{thm:main}}{Theorem 1.1} given \texorpdfstring{\cref{prop:U4-inv-cubic}}{Proposition 1.4}}\label{sec:deduc}
In this section, we convert \cref{prop:U4-inv-cubic} into a density increment using the strategy of Heath-Brown \cite{Hea87} and Szemer\'{e}di \cite{Sze90}. Our treatment is closely modeled on that of Green and Tao \cite{GT09}; the crucial idea is that one may group together a large number of phases before passing to a subprogression.

Throughout this section we will consider functions $f\colon[N']\to[-1,1]$ corresponding to sets and shifted indicators. By Bertrand's postulate, we may find a prime $N$ such that $1024 N' \le N \le 2048 N'$. We may thus embed $[N']$ inside $\mb{Z}/(N\mb{Z})$ and lift $f$ to $\mb{Z}/N\mb{Z}$ (mapping inputs not congruent to an element of $[N']$ to $0$). We define the quintilinear operator
\[\Lambda(f_1,\ldots,f_5) = \mb{E}_{x,y\in\mb{Z}/N\mb{Z}} \prod_{k=1}^{5}f_k(x+(k-1)y)\text{ and }\Lambda(f) = \Lambda(f,f,f,f,f).\]

We now state the key claim for this section, which we will prove in \cref{sub:progression-partition}
\begin{proposition}\label{prop:trichotomy}
Fix a constant $c>0$. There exist $c'>0$ and $C>0$ such that the following holds. Consider a function $f\colon[N']\to [0,1]$ such that $\mb{E}_{n\in [N']}f(n) = \delta>0$ and a prime $N$ such that $1024 N'\le N\le 2048 N'$. Let $M(\delta) = \exp((\log(1/\delta))^C)$. At least one of the following possibilities always holds:
\begin{itemize}
    \item $N'\le\exp(M(\delta))$;
    \item $\big|\Lambda(f) - \Lambda(\delta \mbm{1}_{[N']})\big| \le c \delta^{5}$;
    \item There exists an arithmetic progression $P$ of length at least $N^{1/M(\delta)}$ such that
    \[\mb{E}_{n\in P}f(n)\ge (1+c')\delta.\]
\end{itemize}
\end{proposition}

With this, we can prove the main result.
\begin{proof}[Proof of \cref{thm:main} given \cref{prop:trichotomy}]
Suppose that $A\subseteq [N]$ has size $\delta N$ and has no $5$-term arithmetic progressions. We now perform the density increment strategy using $A_0 = A$, $N_0' = N$, and $\delta_0 = \delta$.

For each $N_i'$ choose a prime $N_i$ between $1024 N_i'\le N_i\le 2048 N_i'$. If $N_i'\le M(\delta_i)$, we immediately terminate. Else note that 
\[\big|\Lambda(\mbm{1}_{A_i}) - \Lambda(\delta_i\mbm{1}_{[N_i']})\big| \ge |\Lambda(\delta_i\mbm{1}_{[N_i']})\big| - |A_i|N_i^{-2}\gg \delta_i^{5}\]
as $A_i$ is free of $5$-term arithmetic progressions. Therefore the third case in the trichotomy must hold and there exists a long progression $P_i$ on which the density of $A_i$ increases by a factor of at least $(1+c')$. Let $A_{i+1}$ be $A_i\cap P_i$ rescaled so that $P_i$ starts at $0$ and has common difference $1$, let $N_{i+1}' = |P_i|$, and let $\delta_{i+1} = |A_i\cap P_i|/|P_i|\ge(1+c')\delta_i$. 

Since the density of a set cannot exceed $1$, we must terminate in at most $O(\log(1/\delta))$ iterations. If we terminate at $i$, we must have
\[N^{{(1/M(\delta))}^{O(\log(1/\delta))}}\le N_i'\le\exp(M(\delta_i))\le\exp(M(\delta)).\]
Rearranging this gives exactly that
\[\log N \le {M(\delta)}^{O(\log(1/\delta))}\]
or, as desired, 
\[\delta \ll \frac{1}{\exp((\log\log N)^{\Omega(1)})}.\qedhere\]
\end{proof}

\subsection{Inverse theorem relative to linear and cubic factors}\label{sub:factors}
We introduce a framework for studying functions satisfying a correlation as given by \cref{prop:U4-inv-cubic}, considering a $\sigma$-algebra (or \emph{factor}) which incorporates the information of the approximate value of our linear and cubic functions on $[N]$. Our treatment closely follows that of Green and Tao \cite{GT09} (and uses elements from Peluse and Prendiville \cite{PP22}).

\begin{definition}\label{def:factor}
We define a \emph{factor} $\mc{B}$ of $\mb{Z}/N\mb{Z}$ to be a partition $\mb{Z}/N\mb{Z} = \bigsqcup_{B\in \mc{B}}B$. We define $\mc{B}(x)$ for $x\in\mb{Z}/N\mb{Z}$ to be the part of $\mc{B}$ that contains $x$.

We say $\mc{B}'$ refines $\mc{B}$ if every part of $\mc{B}$ can be written as a disjoint union of parts of $\mc{B}'$. We define a \emph{join} of a sequence of factors to be the partition (discarding empty parts)
\[\mc{B}_1\vee\cdots\vee\mc{B}_d := \{B_1\cap\cdots\cap B_d\colon B_i\in \mc{B}_i\}.\]

Define a function $\phi\colon S\to\mb{R}/\mb{Z}$ for $S\subseteq\mb{Z}/N\mb{Z}$ to be \emph{irrational} if $\phi$ takes on irrational values. Define the factor $\mc{B}_{\phi,K}$ with respect to $\phi$ of resolution $K$ as the partition
\[((\mb{Z}/N\mb{Z})\setminus S)\sqcup \bigsqcup_{0\le j\le K-1}\{x\in\mb{Z}/N\mb{Z}\colon\snorm{\phi(x) - j/K}_{\mb{R}/\mb{Z}}\le 1/(2K)\}.\] 
(Since $\phi$ is irrational this is indeed a disjoint partition.) We further say that $\phi$ respects a factor $\mc{B}$ if $\mc{B}$ refines $S\sqcup ((\mb{Z}/N\mb{Z})\setminus S)$.

We define a factor of complexity $(d_1,d_2)$ and resolution $K$ via the data of irrational linear functions $\phi_1,\ldots,\phi_{d_1}$ (defined on $\mb{Z}/N\mb{Z}$) and irrational locally cubic functions $\phi_1',\ldots,\phi_{d_2}'$ (defined on subsets of $\mb{Z}/N\mb{Z}$) which respect $\bigvee_{1\le j\le d_1}\mc{B}_{\phi_j,K}$. The associated factor is $\bigvee_{1\le j\le d_1}\mc{B}_{\phi_j,K} \vee \bigvee_{1\le j\le d_1}\mc{B}_{\phi_j',K}$.

Finally, given a factor $\mc{B}$ we define $\Pi_{\mc{B}}f$ by 
\[\Pi_{\mc{B}}f(x) = \mb{E}_{y\in\mc{B}(x)}f(y).\]
\end{definition}
\begin{remark}
We ensure all functions $\phi$ we consider are irrational in order to avoid issues regarding the hitting exactly the boundary of the factor. Note that if a locally cubic function $\phi$ respects a partition $\mc{B}$ then we see that the support set is in the $\sigma$-algebra generated by $\mc{B}$. In particular, we can treat $\phi$ as either ``undefined'' on an atom $\mc{B}$ or as locally cubic on the associated atom.
\end{remark}

We now restate \cref{prop:U4-inv-cubic} in the language of factors.
\begin{lemma}\label{lem:factor-restatement}
There exists $C>0$ such that the following holds. Fix $\eta>0$ and let $f$ be a function $f\colon\mb{Z}/N\mb{Z}\to[0,1]$ such that $\snorm{f}_{U^4(\mb{Z}/N\mb{Z})}\ge\eta$.

Let $M(\eta) = \exp((\log(1/\eta))^C)$ and fix an integer $K\ge \exp(M(\eta))$. Suppose that $N\ge 8K$. There exist $d\le M(\eta)$ and a factor $\mc{B}$ of complexity $(d,1)$ and resolution $K$ such that
\[\snorm{\Pi_{\mc{B}}f}_{L^{1}(\mb{Z}/N\mb{Z})} \ge \exp(-M(\eta)).\]
\end{lemma}
\begin{proof}
By \cref{prop:U4-inv-cubic}, there exist $\rho=1/100$, $S\subseteq (1/N)\mb{Z}$ with $|S|\ll(\log(1/\eta))^{C}$ (say $|S|=d$), and $y\in \mb{Z}/N\mb{Z}$ such that $B=B(S,\rho)$ is a Bohr set and $\phi\colon y+B\to\mb{R}/\mb{Z}$ is a locally cubic phase function such that
\[|\mb{E}_{t\in\mb{Z}/N\mb{Z}}\mbm{1}_{t\in y + B}f(t)e(-\phi(t))|\ge \exp(-(\log(1/\eta))^C/2).\]

Let $C'$ be a sufficiently large constant and assume $K \ge \exp((\log(1/\eta))^{C'})$. Take $\phi_i\colon\mb{Z}/N\mb{Z}\to \mb{R}/\mb{Z}$ to be $\phi_i(x) = a_ix - a_iy + \alpha$ where $\alpha$ is an irrational which is smaller than $(2K)^{-1}$ and $a_i\in S$. 

Note that $y + B$ is exactly the set
\[T_1 = \Big\{x\in\mb{Z}/N\mb{Z}\colon \sup_{a_i\in S}\snorm{a_ix - a_i y}_{\mb{R}/\mb{Z}}\le\rho\Big\}.\]
Note that the set
\[T_2 = \Big\{x\in\mb{Z}/N\mb{Z}\colon\sup_{a_i\in S}\snorm{a_ix - a_iy + \alpha}_{\mb{R}/\mb{Z}}\le (2K)^{-1}\cdot (2\lfloor K \rho\rfloor-1)\Big\}\]
is measurable with respect to the factor $\bigvee_{1\le j\le d}\mc{B}_{\phi_j,K}$. Also, $T_2\subseteq T_1$ since $|\alpha|\le 1/(2K)$ and $(2K)^{-1}\cdot (2\lfloor K \rho\rfloor-1) + (2K)^{-1}\le\rho$. Furthermore we have
\[T_1\setminus T_2 \subseteq \Big\{x\in\mb{Z}/N\mb{Z}\colon\sup_{a_i\in S}\snorm{a_ix - a_i y}_{\mb{R}/\mb{Z}}\in[\rho-2/K,\rho)\Big\}\]
and thus we have
\[|T_2\setminus T_1|\le |S|\cdot(5/K)N= 5dN/K\]
as long as $8K\le N$. Here we have used that $S\subseteq\mb{Z}/N\mb{Z}$ where $N$ is prime.

So, since $C'$ is large with respect to $C$ and $f$ is $1$-bounded we have 
\[|\mb{E}_{t\in\mb{Z}/N\mb{Z}}\mbm{1}_{t\in T_2}f(t)e(-\phi(t))|\gg \exp(-(\log(1/\eta))^C/2).\]
The locally cubic function we will consider is $\phi^{\ast}\colon T_2\to \mb{R}/\mb{Z}$ given by $\phi^{\ast}(x) = \phi(x) + \alpha$. This is well-defined as $T_2\subseteq T_1 = y+B$. Let $\mc{B} = \bigvee_{1\le j\le d}\mc{B}_{\phi_j,K} \vee \mc{B}_{\phi^\ast,K}$ and note that 
\[\sup_{x\in \mb{Z}/N\mb{Z}}\big|e(-\phi(x)) - \Pi_{\mc{B}}[e(-\phi^{\ast}(x))]\big|\le 2\pi/K\]
since $z\mapsto e(z)$ is a $2\pi$-Lipschitz function on $\mb{R}/\mb{Z}$. Therefore 
\[|\mb{E}_{t\in\mb{Z}/N\mb{Z}}f(t)\Pi_{\mc{B}}[e(-\phi^{\ast}(x))]|\ge\exp(-(\log(1/\eta))^C).\]
$\Pi_{\mc{B}}$ is self-adjoint with respect to the standard inner product (see e.g.~\cite[Lemma~4.3(ii)]{PP22}) so
\begin{align*}
\mb{E}_{t\in\mb{Z}/N\mb{Z}}|\Pi_{\mc{B}}[f(t)]|&\ge 
|\mb{E}_{t\in\mb{Z}/N\mb{Z}}\Pi_{\mc{B}}[f(t)]e(-\phi^{\ast}(t))|=|\mb{E}_{t\in\mb{Z}/N\mb{Z}}f(t)\Pi_{\mc{B}}[e(-\phi^{\ast}(t))]|\\
&\ge\exp(-(\log(1/\eta))^C).
\end{align*}
This gives the desired result.
\end{proof}

We now prove an associated Koopman--von Neumann theorem; our proof is closely modeled after \cite[Theorem~5.6]{GT09}.
\begin{lemma}\label{lem:iterated}
There exists $C>0$ such that the following holds. Fix $\eta>0$, let $f$ be a function $f\colon\mb{Z}/N\mb{Z}\to [0,1]$, and let $\mc{B}^{\ast}$ denote an initial factor. 

Let $M(\eta) = \exp(\log(1/\eta)^C)$ and let $N\ge 10 M(\eta)$. There exist $d_1,d_2\le M(\eta)$ such that the following holds. There exists an integer $K\le M(\eta)$ and a factor of $\mc{B}'$ of complexity $(d_1,d_2)$ and resolution $K$ such that if $\mc{B} = \mc{B}^{\ast} \vee \mc{B}'$ then
\[\snorm{f - \Pi_{\mc{B}}f}_{U^4(\mb{Z}/N\mb{Z})} \le \eta.\]
\end{lemma}
\begin{proof}
We create the factor $\mc{B}'$ as a join given by applying \cref{lem:factor-restatement} iteratively. Set $\mc{B}_0 = \mc{B}^{\ast}$. 
\begin{itemize}
    \item If $\snorm{f-\Pi_{\mc{B}_i} f}_{U^4(\mb{Z}/N\mb{Z})}\le \eta$, we terminate. Set $i^{\ast} = i$ and output $\mc{B}'=\bigvee_{0\le i'<i^\ast}\mc{B}_{i'}'$.
    \item If $\snorm{f-\Pi_{\mc{B}_i} f}_{U^4(\mb{Z}/N\mb{Z})}\ge \eta$, then set $K = \lceil \exp(M(\eta))\rceil$. By \cref{lem:factor-restatement} there exists $\mc{B}_i'$ such that 
    \[\snorm{\Pi_{\mc{B}_i'}(f-\Pi_{\mc{B}_i} f)}_{L^1(\mb{Z}/N\mb{Z})}\ge \exp(-(\log(1/\eta))^{C'}).\]
    Here $\mc{B}_i' = \bigvee_{1\le j\le d_i}\mc{B}_{\phi_j^{(i)},K} \vee \mc{B}_{\phi_i',K}$ where $\phi_j^{(i)}$ (for $j\le d_i\le\exp((\log(1/\eta))^{O(1)}))$) are irrational linear functions on $\mb{Z}/N\mb{Z}$ and $\phi_i'$ is a locally cubic function respecting the factor $\bigvee_{1\le j\le d}\mc{B}_{\phi_j,K}$. We set $\mc{B}_{i+1} = \mc{B}_i\vee \mc{B}_i'$.
\end{itemize}
Note that $\mc{B}=\mc{B}^\ast\vee\mc{B}'=\mc{B}_{i^\ast}$. Observe that
\begin{align*}
\snorm{\Pi_{\mc{B}_i'}(f-\Pi_{\mc{B}_i}f)}_{L^1(\mb{Z}/N\mb{Z})}&\le\snorm{\Pi_{\mc{B}_i'}(f-\Pi_{\mc{B}_i}f)}_{L^2(\mb{Z}/N\mb{Z})}=\snorm{\Pi_{\mc{B}_i'}\Pi_{\mc{B}_{i+1}}(f-\Pi_{\mc{B}_i}f)}_{L^2(\mb{Z}/N\mb{Z})}\\
&\le\snorm{\Pi_{\mc{B}_{i+1}}(f-\Pi_{\mc{B}_i}f)}_{L^2(\mb{Z}/N\mb{Z})}=\snorm{\Pi_{\mc{B}_{i+1}}f-\Pi_{\mc{B}_i}f}_{L^2(\mb{Z}/N\mb{Z})}\\
&=(\snorm{\Pi_{\mc{B}_{i+1}}f}_{L^2(\mb{Z}/N\mb{Z})}^2-\snorm{\Pi_{\mc{B}_i}f}_{L^2(\mb{Z}/N\mb{Z})}^2)^{1/2}.
\end{align*}
The final equality is the Pythagorean theorem with respect to projections (this follows from e.g.~\cite[Lemma~4.3(iv)]{PP22}). 

Thus $\snorm{\Pi_{\mc{B}_{i+1}f}}_{L^2(\mb{Z}/N\mb{Z})}^2-\snorm{\Pi_{\mc{B}_i}f}_{L^2(\mb{Z}/N\mb{Z})}^2\ge \exp(-(\log(1/\eta))^{O(1)})$ and hence the iteration lasts only $\exp((\log(1/\eta))^{O(1)})$ steps. Therefore the final $\mc{B}'$ is generated by at most $\exp((\log(1/\eta))^{O(1)})$ linear functions and a similar number of locally cubic functions. This completes the proof. 
\end{proof}

\subsection{Density increment onto cubic Bohr set}\label{sub:dense-Bohr}
We next need that $\Lambda$ is controlled by the $U^4$-norm and the $L^{1}$-norm. The proof is by now standard and hence is omitted (see \cite[Lemma~3.2]{GT09} and \cite[Theorem~3.2]{Gow01a}).
\begin{lemma}\label{lem:bounds}
Let $N$ be prime and $f_i\colon\mb{Z}/N\mb{Z}\to \mb{C}$. Then we have:
\begin{align*}
|\Lambda(f_1,f_2,f_3,f_4,f_5)|&\le \min_{1\le i\le 5}\snorm{f_i}_{L^1(\mb{Z}/N\mb{Z})}\prod_{j\neq i}\snorm{f_j}_{L^\infty(\mb{Z}/N\mb{Z})},\\
|\Lambda(f_1,f_2,f_3,f_4,f_5)|&\le \min_{1\le i \le 5}\snorm{f_i}_{U^4(\mb{Z}/N\mb{Z})}\prod_{j\neq i}\snorm{f_j}_{L^\infty(\mb{Z}/N\mb{Z})}.
\end{align*}
\end{lemma}

Given this we may now prove that there exists a density increment onto a piece of the partition given by \cref{lem:iterated}. The proof is identical to \cite[Lemma~5.8]{GT09}.
\begin{proposition}\label{prop:cubic-bohr-increment}
Fix $c>0$. There exist $C>0$ and $c'> 0$ such that the following holds. Given a function $f\colon[N']\to[0,1]$ such that $\mb{E}_{n\in[N']}f(n) = \delta>0$, extend it by $0$s to a function on $\mb{Z}/N\mb{Z}$ where $N$ is prime with $1024N'\le N\le 2048N'$. Suppose that
\[|\Lambda(f) - \Lambda(\delta \mbm{1}_{[N']})|\ge c\delta^{5}.\]

Let $\mc{B}^{\ast} = [N']\sqcup ((\mb{Z}/N\mb{Z})\setminus [N'])$ be an initial factor. Let $M(\delta) = \exp((\log(1/\delta))^C)$ and suppose that $N\ge M(\delta)$. Then there exist $d_1,d_2,K\le M(\delta)$, a factor $\mc{B}$ of complexity $(d_1,d_2)$ and resolution $K$, and an atom $\Omega^\ast$ of $\mc{B}\vee\mc{B}^{\ast}$ such that
\begin{itemize}
    \item $\mb{P}_{n\in \mb{Z}/N\mb{Z}}[n\in\Omega^\ast]\ge \exp(-\exp((\log(1/\delta))^C))$;
    \item $\mb{E}[f(n)|n\in\Omega^\ast]\ge (1+c')\delta$.
\end{itemize}
\end{proposition}
\begin{proof}
We may clearly assume $\delta$ is sufficiently small. By \cref{lem:iterated}, there exists a factor $\mc{B}$ of complexity $(d_1,d_2)$ and resolution $K$ with $d_1,d_2,K\le \exp((\log(1/\delta))^{O(1)})$ such that 
\[\snorm{f - \Pi_{\mc{B}\vee \mc{B}^{\ast}}f}_{U^4(\mb{Z}/N\mb{Z})}\le c\delta^{5}/10.\]
By telescoping and the second part of \cref{lem:bounds} we have 
\[\big|\Lambda(f) - \Lambda(\Pi_{\mc{B}\vee \mc{B}^{\ast}}f)\big|\le 5\snorm{f-\Pi_{\mc{B}\vee \mc{B}^{\ast}}f}_{U^4(\mb{Z}/N\mb{Z})}\le c\delta^{5}/2.\]
Therefore we have 
\[|\Lambda(\Pi_{\mc{B}\vee \mc{B}^{\ast}}f) - \Lambda(\delta \mbm{1}_{[N']})|\ge c\delta^{5}/2.\]

Let $\mc{B}' = \mc{B}\vee \mc{B}^{\ast}$ and $g = \min(\Pi_{\mc{B}'}f,(1+c')\delta)$ and assume that $c'\le \max(c,1)/10^{5}$. Let $\Omega' = \{x\in\mb{Z}/N\mb{Z}\colon g(x)\neq \Pi_{\mc{B}'}f(x)\}$ and notice that 
\begin{align*}
|\Lambda(\Pi_{\mc{B}'}f) - \Lambda(g)|&\le 5\snorm{\Pi_{\mc{B}'}f - g}_{L^1(\mb{Z}/N\mb{Z})} \le 5 \mb{P}[n\in \Omega']\\
|\Lambda(\delta\mbm{1}_{[N']}) - \Lambda(g)|&\le 5(1+c')^4\delta^{4}\snorm{\delta\mbm{1}_{[N']}-g}_{L^1(\mb{Z}/N\mb{Z})}\\
\snorm{g - \delta \mbm{1}_{[N']}}_{L^1(\mb{Z}/N\mb{Z})}&\le \mb{P}[n\in \Omega'] + \snorm{\delta \mbm{1}_{[N']} - \Pi_{\mc{B}'}f}_{L^1(\mb{Z}/N\mb{Z})}.
\end{align*}
The first and second inequality follow from the first part of \cref{lem:bounds} while the final inequality follows from the triangle inequality. 

Suppose that the proposition is false. Note that there are at most $2(2K)^{d_1+d_2}$ atoms of $\mc{B}'$ and define an atom to be small if it has measure at most $(2K)^{-d_1-d_2}\cdot (c\delta^{5})/40$. We therefore have that on every atom which is not small,
\[\Pi_{\mc{B}'}f\le (1+c')\delta\]
(else we take $\Omega^\ast$ to be that atom) and therefore $\mb{P}[n\in \Omega']\le c\delta^{5}/20$.  Thus by the first inequality above and the triangle inequality we find
\[|\Lambda(\delta\mbm{1}_{[N']}) - \Lambda(g)|\ge|\Lambda(\delta\mbm{1}_{[N']}) - \Lambda(\Pi_{\mc{B}'}f)|-5\mb{P}[n\in\Omega^\ast]\ge c\delta^{5}/4.\]
By the second inequality, this implies that 
\[\snorm{\delta\mbm{1}_{[N']}-g}_{L^1(\mb{Z}/N\mb{Z})}\ge c\delta/100.\]
By the third inequality we have that 
\[\snorm{\delta \mbm{1}_{[N']} - \Pi_{\mc{B'}}f}_{L^1(\mb{Z}/N\mb{Z})}\ge c\delta/200.\]
For a function $h$ let $h_{+} = \max(h,0)$. For mean zero functions, $\snorm{h}_{L^1(\mb{Z}/N\mb{Z})} = 2\snorm{h_{+}}_{L^1(\mb{Z}/N\mb{Z})}$. This allows us to derive the contradiction since 
\begin{align*}
\snorm{\delta \mbm{1}_{[N']} - \Pi_{\mc{B'}}f}_{L^1(\mb{Z}/N\mb{Z})} &= 2\snorm{(\Pi_{\mc{B'}}f-\delta \mbm{1}_{[N']})_{+}}_{L^1(\mb{Z}/N\mb{Z})}\\
&\le 2c'\delta + 2((2K)^{-d_1-d_2}\cdot (c\delta^{5})/40)(2(2K)^{d_1+d_2})\le 4c'\delta<c\delta/200.\qedhere
\end{align*}
\end{proof}

\subsection{Density increment onto subprogression}\label{sub:progression-partition}
We now finish the argument by partitioning factors in our cubic Bohr partition into long progressions. This is a technical modification of an argument of Green and Tao \cite{GT09} for multiple quadratics (the case of a single polynomial is handled in Gowers \cite{Gow01a}); as input they must provide, in the case of quadratics, an explicit implicit constant for a result of Schmidt \cite{Sch77} which provides recurrence for multiple polynomials mod $1$ simultaneously. We require the analogous result for arbitrary degrees.

\begin{proposition}\label{prop:schmidt}
Fix an integer $k\ge 1$. There exist $c = c(k)>0$ such that the following holds. Let $\alpha_1,\ldots,\alpha_d$ be real numbers. Then 
\[\min_{1\le n\le N}\max_{1\le i\le d}\snorm{\alpha_i n^k}_{\mb{R}/\mb{Z}}\ll_k dN^{-c/d^2}.\]
\end{proposition}
We defer the proof of this input to \cref{sec:schmidt}; it is an essentially verbatim copy of \cite[Appendix~A]{GT09}.

\begin{proof}[Proof of \cref{prop:trichotomy}]
\textbf{Step 1: Setup.}
Let $M=\exp((\log(1/\delta))^C)$ where $C>0$ will be chosen later. Let us assume $N'>\exp(M(\delta))$ and $|\Lambda(f)-\Lambda(\delta\mbm{1}_{[N']})|>c\delta^5$. Our aim is to find a progression of length at least $N^{1/M(\delta)}$ where our density is incremented

Let $\mc{B}^{\ast} = [N']\sqcup ((\mb{Z}/N\mb{Z})\setminus [N'])$ be an initial factor. Let $M'(\delta) = \exp((\log(1/\delta))^{C'})$ where $C'$ is chosen appropriately so as to apply \cref{prop:cubic-bohr-increment}. By assumption we have $N\ge M'(\delta)$. So there exist $d_1,d_2,K\le M'(\delta)$, a factor $\mc{B}$ of complexity $(d_1,d_2)$ and resolution $K$, and an atom $\Omega^\ast$ of $\mc{B}\vee\mc{B}^{\ast}$ such that
\begin{itemize}
    \item $\mb{P}_{n\in \mb{Z}/N\mb{Z}}[n\in\Omega^\ast]\ge\exp(-\exp((\log(1/\delta))^{C'}))$;
    \item $\mb{E}[f(n)|n\in\Omega^\ast]\ge(1+2c')\delta$.
\end{itemize}
Now our aim is to partition $\Omega^\ast$ into few arithmetic progressions, namely at most $X=7^dN^{1-\Omega(1/d^7)}$ progressions where $d=M'(\delta)$. The number of elements of $\Omega^\ast$ in progressions of length at most $c' \delta |\Omega^\ast|/X$ is at most $c'\delta|\Omega^\ast|$. Letting the union of short progressions be $\Omega'$, we have that
\[\mb{E}[f(n)|n\in\Omega^\ast\setminus \Omega']\ge (1+c')\delta.\]
Therefore there is a progression of length longer than $c' \delta |\Omega^\ast|/X$ with density at least $(1+c')\delta$. We have
\[c'\delta|\Omega^\ast|/X\ge N^{1/M(\delta)}\]
due to our assumption $N'>\exp(M(\delta))$, assuming $C>0$ is chosen sufficiently large. (This is where the dependence in the first item of \cref{prop:trichotomy} comes from.)

We may write
\[\Omega^\ast=[N']\cap\bigcap_{i\in I_1}\{x\colon\snorm{\phi_i(x)-j_i/K}_{\mb{R}/\mb{Z}}\le1/(2K)\}\cap\bigcap_{i\in I_2}\{x\colon\snorm{\varphi_i(x)-j_i/K}_{\mb{R}/\mb{Z}}\le1/(2K)\}\]
where $\phi_i$ are linear functions on $\mb{Z}/N\mb{Z}$ and $\varphi_i\colon T\to\mb{R}$ are locally cubic functions on
\[T:=[N']\cap\bigcap_{i\in I_1}\{x\colon\snorm{\phi_i(x)-j_i/K}_{\mb{R}/\mb{Z}}\le1/(2K)\}.\]
We additionally have $|I_1|,|I_2|,K\le M'(\delta)=:d$. (The fact that $\Omega^\ast\subseteq[N']$ is evident from $\mb{E}[f(n)|n\in\Omega^\ast]>0$, recalling that $f$ is extended by $0$s in \cref{prop:cubic-bohr-increment}.)

\noindent\textbf{Step 2: Partitioning $T$ into progressions where the cubic phases are near-constant.}
Our goal will be to partition $T$ into subprogressions on which each $\varphi_i$ for $i\in I_2$ is roughly constant $\mod 1$. In fact, we will find a collection of progressions, say of the form $\{a+bn\colon n\in[L]\}$, so that for $n\in[L]$ we have $\varphi_i(a+bn)=\kappa+P_i^{(a,b)}(n)\mod 1$, where $P_i^{(a,b)}$ is a real polynomial of degree at most $3$ with coefficients of size $O(L^{-4})$.

We first partition
\[T=\bigsqcup_{j\in J_1}T_j\]
where $T_j$ are progressions and $|J_1|\ll 2^{|I_1|}N^{1-1/(|I_1|+1)}$. To do this, we may use \cite[Proposition~6.3]{GT09} (which is a simple application of Kronecker approximation, or \cref{prop:schmidt} for $k=1$).

Since $T_j$ is a progression, we may write $T_j=\{a_j+d_j,\ldots,a_j+M_jd_j\}$ where $M_j=|T_j|$. Since each $\varphi_i$ is locally cubic on this progression, we may write
\[\varphi_i(a_j+d_jn)=\alpha_i^{(j)}n^3+\beta_i^{(j)}n^2+\gamma_i^{(j)}n+\delta_i^{(j)} \imod 1\]
for $i\in I_2$ and $n\in[M_j]$ where $j\in J_1$. Next, we perform three intermediate partitions of $T_j$ in order to iteratively ``reduce the degree'' of the function $\varphi_i$ until we see that it is roughly constant on our subprogressions.

First, choose $1\le r_j\le M_j^{1/3}$ so that
\[\max_{i\in I_2}\snorm{\alpha_i^{(j)}r_j^3}_{\mb{R}/\mb{Z}}\ll dM_j^{-c(3)/(3d^2)}.\]
Then we can partition $[M_j]$ into at most $M_j^{1-c(3)/(120d^2)}$ progressions of difference $r_j$ and length roughly $M_j^{c(3)/(120d^2)}$. This corresponds to a partition
\[T_j=\bigsqcup_{k_3\in D_3}T_{j,k_3}\]
with $|D_3|\le M_j^{1-c(3)/(120d^2)}$ and $|T_{j,k_3}|\sim M_j^{c(3)/(120d^2)}$. Furthermore, we have
\[T_{j,k_3}=\{a_j+(b_{j,k_3}+r_jt)d_j\colon t\in[|T_{j,k_3}|]\}\]
for some integer $b_{j,k_3}$. We have
\begin{align*}
\varphi_i(a_j+(b_{j,k_3}+r_jn)d_j)&=\alpha_i^{(j)}(b_{j,k_3}+r_jn)^3+\beta_i^{(j)}(b_{j,k_3}+r_jn)^2+\gamma_i^{(j)}(b_{j,k_3}+r_jn)+\delta_i^{(j)}\\
&=\alpha_i^{(j,k_3)}n^2+\beta_i^{(j,k_3)}n+\gamma_i^{(j,k_3)}+\{\alpha_i^{(j)}r_j^3\}n^3 \imod 1.
\end{align*}
By construction, $\{\alpha_i^{(j)}r_j^3\}$ is close to $0\imod 1$ (namely, it is $\ll dM_j^{-c(3)/(3d^2)}\ll|T_{j,k_3}|^{-30}$).

We can thus consider the quadratic function obtained by removing this part, and repeat the same process on each $T_{j,k_3}$. We obtain iterative decompositions
\[T_{j,k_3}=\bigsqcup_{k_2\in D_2^{(k_3)}}T_{j,k_3,k_2},\quad T_{j,k_3,k_2}=\bigsqcup_{k_1\in D_1^{(k_3,k_2)}}T_{j,k_3,k_2,k_1}\]
where:
\begin{align*}
|D_2^{(k_3)}|\le|T_{j,k_3}|^{1-c(2)/(120d^2)},\quad |D_1^{(k_3,k_2)}|\le|T_{j,k_3,k_2}|^{1-c(1)/(120d^2)}\\
|T_{j,k_3,k_2}|\sim|T_{j,k_3}|^{c(2)/(120d^2)},\quad|T_{j,k_3,k_2,k_1}|\sim|T_{j,k_3,k_2}|^{c(1)/(120d^2)}.
\end{align*}
Additionally, we will find that if $T_{j,k_3,k_2,k_1}=\{a+bn\colon n\in[|T_{j,k_3,k_2,k_1}|]\}$ then
\[\varphi_i(a+bn)=\alpha_i^{(j,k_3,k_2,k_1)}n^3+\beta_i^{(j,k_3,k_2,k_1)}n^2+\gamma_i^{(j,k_3,k_2,k_1)}n+\kappa_i^{(j,k_3,k_2,k_1)} \imod 1\]
where
\[|\alpha_i^{(j,k_3,k_2,k_1)}|+|\beta_i^{(j,k_3,k_2,k_1)}|+|\gamma_i^{(j,k_3,k_2,k_1)}|\ll|T_{j,k_3,k_2,k_1}|^{-4}.\]

The number of such $T_{j,k_3,k_2,k_1}$ partitioning $T_j$ is in total at most
\[M_j^{1-c(3)/(12d^2)}(M_j^{c(3)/(12d^2)})^{1-c(2)/(12d^2)}\big((M_j^{c(3)/(12d^2)})^{c(2)/(12d^2)}\big)^{1-c(1)/(12d^2)}\le M_j^{1-\Omega(1/d^6)}.\]
(In the rare cases of any $j$ such that $M_j\le\exp(O(d^6))$ we may simply partition $M_j$ into progressions of length $1$ appropriately.)

Finally, choosing appropriate value $p=\Omega(1/d^6)$, the total number of subprogressions is
\[\sum_{j\in J_1}M_j^{1-p}\le|J_1|^p\bigg(\sum_{j\in J_1}M_j\bigg)^{1-p}\le|J_1|^pN^{1-p}\le N^{1-\Omega(1/d^7)}.\]

\noindent\textbf{Step 3: Partitioning $\Omega^\ast$ into few progressions.}
From the previous part, we have a partition of $T$ into at most $N^{1-\Omega(1/d^7)}$ progressions of various lengths so that for a progression $\{a+bn\colon n\in[L]\}$ that appears, we have
\[\varphi_i(a+bn)=\kappa+P_i^{(a,b)}(n) \imod 1\]
where $P_i^{(a,b)}$ is a real polynomial of degree at most $3$ with coefficients of size $O(L^{-4})$. Recalling $\Omega^\ast\subseteq T$, we can now intersect these progressions with $\Omega^\ast$. Note that $\Omega^\ast$ merely imposes a condition on each $\varphi_i$ for $i\in I_2$. Since each $P_i^{(a,b)}(n)$ is at most degree $3$, within the range $n\in[L]$ it hits the cutoffs for the condition for $\varphi_i$ at most $2\cdot 3=6$ times in the real numbers. Since $P_i^{(a,b)}(n)$ has small coefficients (so is small on $n\in[L]$), there is no rounding wraparound and the same holds for its image $\imod 1$ over the domain $n\in[L]$.

The upshot is that each progression is cut into at most $7$ pieces by the condition on $\varphi_i$ defining $\Omega^\ast$ for each $i\in I_2$. This gives at most $7^{|I_2|}\le 7^d$ total pieces.

Thus, we have a partition of $\Omega^\ast$ into at most $7^dN^{1-\Omega(1/d^7)}$ pieces, which combined with the argument above completes the proof.
\end{proof}

\bibliographystyle{amsplain0.bst}
\bibliography{main.bib}

\appendix
\section{Conventions regarding nilsequences and effective equidistribution}\label{sec:nilmanifolds}

We begin this appendix by giving the precise definition of the complexity of a nilmanifold; this definition is exactly as in \cite[Definition~6.1]{TT21}.
\begin{definition}\label{def:nilmanifold}
Let $s\ge 1$ be an integer and let $K>0$. A \emph{filtered nilmanifold $G/\Gamma$ of degree $s$ and complexity at most $K$} consists of the following:
\begin{itemize}
    \item a nilpotent, connected, and simply connected Lie group $G$ of dimension $m$, which can be identified with its Lie algebra $\log G$ via the exponential map $\exp\colon\log G\to G$;
    \item a filtration $G_{\bullet} = (G_i)_{i\ge 0}$ of closed connected subgroups $G_i$ of $G$ with 
    \[G = G_0 = G_1\geqslant G_1\geqslant \cdots\geqslant G_s\geqslant G_{s+1} = \mr{Id}_G\]
    such that $[G_i,G_j]\subseteq G_{i+j}$ for all $i,j\ge 0$ (equivalently $[\log G_i, \log G_j]\subseteq \log G_{i+j}$);
    \item a discrete cocompact subgroup $\Gamma$ of $G$; and
    \item a linear basis $\mathcal{X}=\{X_1,\ldots, X_{m}\}$ of $\log G$, known as a \emph{Mal'cev basis}.
\end{itemize}
We, furthermore, require that this data obeys the following conditions:
\begin{enumerate}
    \item for $1\le i,j\le m$, one has Lie algebra relations
    \[[X_i,X_j] = \sum_{i,j<k\le m}c_{ijk}X_k\]
    for rational numbers $c_{ijk}$ of height at most $K$ (we will often refer to these as the \emph{Lie bracket structure constants});
    \item for each $1\le i\le s$, the Lie algebra $\log G_i$ is spanned by $\{X_j\colon m-\dim(G_i)<j\le m\}$; and
    \item the subgroup $\Gamma$ consists of all elements of the form $\exp(t_1X_1)\cdots\exp(t_{m}X_{m})$ with $t_i\in \mb{Z}$.
\end{enumerate}
\end{definition}
We note that the conditions imply $[G,G_s]=\mr{Id}_G$, i.e., $G_s$ is contained in the center of $G$ (commutes with every element).

Next, we will define polynomial sequences in filtered nilpotent groups. This concrete definition is equivalent (by~\cite[Lemma 6.7]{GT12}) to the one given in~\cite{GT12}.
\begin{definition}\label{def:polyseq}
We adopt the conventions of \cref{def:nilmanifold}. Let $G$ be a filtered nilpotent group of degree $s$. A function $g\colon\mb{Z}\to G$ is a \emph{polynomial sequence} if there exist elements $g_i\in G_{i}$ for $i=0,\ldots,s$ such that
\begin{equation*}
g(n)=g_0g_1^{\binom{n}{1}}\cdots g_s^{\binom{n}{s}},
\end{equation*}
where $\binom{n}{i}=\frac{1}{i!}\prod_{j=0}^{i-1}(n-j)$, for all $n\in\mb{Z}$. We say a polynomial sequence $g(n)$ is $N$-periodic with respect to a lattice $\Gamma$ if \[g(n)g(n+N)^{-1}\in \Gamma\] for all $n\in \mb{Z}$.
\end{definition}
We will denote the set of polynomial sequences $g\colon\mb{Z}\to G$ relative to the filtration $G_\bullet$ of $G$ by $\on{poly}(\mb{Z},G_\bullet)$. It turns out that $\on{poly}(\mb{Z},G_\bullet)$ is a group under the natural multiplication of sequences--this is due to Lazard~\cite{Lazard54} and Leibman~\cite{Leibman98,Leibman02}.

Now we can define Mal'cev coordinates, the explicit metrics on $G$ and $G/\Gamma$ used in our work, and the precise definition of the Lipschitz norm of functions on $G/\Gamma$. These definitions are exactly as in \cite[Appendix~A]{GT12}.
\begin{definition}\label{def:Lip}
We adopt the conventions of \cref{def:nilmanifold}. Given a Mal'cev basis $\mc{X}$ and $g\in G$, there exists $(t_1,\ldots,t_m)\in \mb{R}^{m}$ such that 
\[g = \exp(t_1X_1 + t_2X_2+\ldots t_mX_m).\]
We define \emph{Mal'cev coordinates of first kind $\psi_{\exp}=\psi_{\exp,\mathcal{X}}\colon G\to\mb{R}^m$ for $g$ relative to $\mc{X}$} by 
\[\psi_{\exp}(g) := (t_1,\ldots,t_m).\]
Given $g\in G$ there also exists $(u_1,\ldots,u_m)\in \mb{R}^m$ such that 
\[g = \exp(u_1X_1)\cdots\exp(u_mX_m),\]
and we define the \emph{Mal'cev coordinates of second kind $\psi=\psi_{\mathcal{X}}\colon G\to\mb{R}^m$ for $g$ relative to $\mc{X}$} by
\[\psi(g) := (u_1,\ldots, u_m).\]
We then define a metric $d = d_{\mc{X}}$ on $G$ by
\[d(x,y) := \inf\bigg\{\sum_{i=1}^{n}\min(\snorm{\psi(x_ix_{i+1}^{-1})},\snorm{\psi(x_{i+1}x_{i}^{-1})})\colon n\in\mb{N}, x_1,\ldots,x_{n+1}\in G, x_1 = x, x_{n+1} = y\bigg\},\]
where $\snorm{\cdot}$ denotes the $\ell^\infty$-norm on $\mb{R}^m$, and define a metric on $G/\Gamma$ by
\[d(x\Gamma,y\Gamma) = \inf_{\gamma,\gamma'\in\Gamma}d(x\gamma,y\gamma').\]
Furthermore, for any function $F\colon G/\Gamma\to\mb{C}$, we define 
\[\snorm{F}_{\mr{Lip}} := \snorm{F}_{\infty} + \sup_{\substack{x,y\in G/\Gamma \\ x\neq y}}\frac{|F(x)-F(y)|}{d(x,y)}.\]
\end{definition}

We recall the notion of a horizontal character and the notion of a function $F$ having a vertical frequency; our definitions are exactly as in \cite[Definitions~1.5,~3.3,~3.4,~3.5]{GT12}.
\begin{definition}\label{def:characters}
Given a filtered nilmanifold $G/\Gamma$, the \emph{horizontal torus} is defined to be \[(G/\Gamma)_{\mr{ab}}:=G/[G,G]\Gamma.\] A \emph{horizontal character} is a continuous homomorphism $\eta\colon G\to\mb{R}/\mb{Z}$ that annihilates $\Gamma$; such characters may be equivalently viewed as characters on the horizontal torus. A horizontal character is \emph{nontrivial} if it is not identically zero. 

Furthermore, if the nilmanifold $G/\Gamma$ has degree $s$, the vertical torus is defined to be \[G_s/(G_s\cap \Gamma).\] A \emph{vertical character} is a continuous homomorphism $\xi\colon G_s\to\mb{R}/\mb{Z}$ that annihilates $\Gamma\cap G_s$. Setting $m_s = \dim G_s$, one may use the last $m_s$ coordinates of the Mal'cev coordinate map to identify $G_s$ and $G_s/(G_s\cap \Gamma)$ with $\mb{R}^{m_s}$ and $\mb{R}^{m_s}/\mb{Z}^{m_s}$, respectively. Thus, we may identify any vertical character $\xi$ with a unique $k\in \mb{Z}^{m_s}$ such that $\xi(x) = k\cdot x$ under this identification $G_s/(\Gamma\cap G_s) \cong \mb{R}^{m_s}/\mb{Z}^{m_s}$. We refer to $k$ as the \emph{frequency} of the character $\xi$, we write $|\xi| :=\snorm{k}_{\infty}$ to denote the magnitude of the frequency $\xi$, and say that a function $F\colon G/\Gamma\to\mb{C}$ \emph{has a vertical frequency $\xi$} if 
\[F(g_s\cdot x) = e(\xi(g_s))F(x)\]
for all $g_s\in G_s$ and $x\in G/\Gamma$.
\end{definition}

\section{Explicit dimension dependence for Schmidt's polynomial recurrence}\label{sec:schmidt}
We now prove the explicit version of Schmidt's polynomial recurrence (\cref{prop:schmidt}). The proof we provide is essentially verbatim from \cite[Appendix~A]{GT09}. For those familiar with the proof presented there, the use of theta functions is unrelated to the fact that one is finding small fractional parts of quadratic polynomials. Roughly speaking, the theta function is only used as a smooth approximation of the neighborhood of points in a lattice which has nice Fourier properties.

\begin{definition}\label{def:theta}
Suppose that $\Lambda$ is a lattice of full rank in $\mb{R}^d$. The dual lattice, denoted $\Lambda^\ast$, is $\Lambda^{\ast} = \{\xi\in \mb{R}^d\colon\xi\cdot m\in \mb{Z}\text{ for all } m\in \Lambda\}$. For any $t>0$ and $x\in \mb{R}^d$, we define 
\[\Theta_{\Lambda}(t,x) := \sum_{m\in \Lambda}e^{-\pi t\snorm{x-m}_2^2}.\]
Finally define 
\[A_{\Lambda} := \Theta_{\Lambda^{\ast}}(1,\vec{0}) = \sum_{\xi\in \Lambda^{\ast}}e^{-\pi\snorm{\xi}_2^2} = \det(\Lambda)\sum_{m\in \Lambda}e^{-\pi \snorm{m}_2^2}.\]
\end{definition}

We next need a version of Weyl's inequality; we take the statement from \cite[Proposition~4.3]{GT12}.
\begin{proposition}\label{prop:weyl}
There exists $C = C(k)>0$ such that the following holds. Suppose that $g\colon\mb{Z}\to\mb{R}$ is a polynomial of degree $k$ and $\delta\in(0,1/2)$. If $N\ge\delta^{-C}$ and 
\[\big|\mb{E}_{n\in [N]}e(g(n))\big|\ge \delta\]
then there exists a positive integer $\ell$ such that $\ell\le\delta^{-C}$ and
\[\snorm{\ell g}_{C^{\infty}[N]}\le\delta^{-C}.\]
\end{proposition}
\begin{remark}
For a polynomial $P(n)=a_0+\cdots+a_kn^k$, we have $\snorm{P}_{C^\infty[N]}=\max_{0<i\le k}N^i\snorm{a_i}_{\mb{R}/\mb{Z}}$.
\end{remark}

A crucial object of study for $\alpha\in\mb{R}^d$ will be
\[F_{\Lambda,\alpha}(N) := \det(\Lambda)\mb{E}_{|n|\le N}\Theta_{\Lambda}(1,n^k\alpha) = \sum_{\xi\in \Lambda^{\ast}}e^{-\pi \snorm{\xi}_2^2}\mb{E}_{|n|\le N}e(n^k\xi\cdot \alpha);\]
the final equality is a consequence of the Poisson summation formula (cf.~\cite[(A.3),~(A.5)]{GT09}).

The following appears as \cite[Lemma~A.5~(i),(ii),(iii)]{GT09} except with trivial modification for the differing degree. 
\begin{lemma}\label{lem:stability}
Let $\Lambda$ be a lattice of full rank in $\mb{R}^d$, let $\alpha\in \mb{R}^d$, and $N>0$. We have the following properties:
\begin{itemize}
    \item (Contraction on $N$) For any $c\in (10/N,1)$, we have $F_{\Lambda,\alpha}(N)\gg cF_{\Lambda,\alpha}(cN)$;
    \item (Dilation of $\alpha$) For any integer $q\ge 1$, we have $F_{\Lambda,\alpha}(N)\gg\frac{1}{q} F_{\Lambda,q^k\alpha}(N/q)$;
    \item (Stability) Suppose that $\snorm{\wt{\alpha}-\alpha}_2\le \eps N^{-k}$ with $\eps\in (0,1)$. Then $F_{\Lambda,\alpha}(N)\gg F_{(1+\eps)\Lambda,(1+\eps)\wt{\alpha}}(N)$.
\end{itemize}
\end{lemma}
\begin{proof}
We prove each of the items in turn. For the first item, note that $\Theta_\Lambda>0$ so
\[F_{\Lambda,\alpha}(N)=\det(\Lambda)\frac{\sum_{|n|\le N}\Theta_{\Lambda}(1,n^k\alpha)}{2\lfloor N\rfloor+1}\ge \det(\Lambda)\frac{\sum_{|n|\le cN}\Theta_{\Lambda}(1,n^k\alpha)}{2\lfloor N\rfloor+1}\ge\frac{c F_{\Lambda,\alpha}(cN)}{2}.\]
For the second item, if $q>N$ note that 
\[F_{\Lambda,\alpha}(N)=\det(\Lambda)\mb{E}_{|n|\le N}\Theta_{\Lambda}(1,n^k\alpha)\ge \det(\Lambda)\frac{\Theta(1,\vec{0})}{2N+1}\ge \frac{F_{\Lambda,q^k\alpha}(N/q)}{3q}.\]
For $q\le N$, we have
\[F_{\Lambda,\alpha}(N)=\det(\Lambda)\frac{\sum_{|n|\le N}\Theta_{\Lambda}(1,n^k\alpha)}{2\lfloor N\rfloor+1}\ge \det(\Lambda)\frac{\sum_{|n|\le N/q}\Theta_{\Lambda}(1,n^k\cdot q^k\alpha)}{2\lfloor N\rfloor+1}\ge \frac{F_{\Lambda,q^k\alpha}(N/q)}{4q}.\]
We now handle the final item; let $X := \snorm{n^k\alpha -m}_2$, $\wt{X} = \snorm{n^k\wt{\alpha}-m}_2$, and note that $|X-\wt{X}|\le \snorm{n^k(\alpha-\wt{\alpha})}_2\le \eps$. We have that $4\pi + \pi(1+\eps)^2\wt{X}^2\ge \pi X^2$ for all $X\ge 0$ and $\wt{X}\in[X\pm\eps]$; the inequality is trivial for $X\le 2$ and for $X\ge 2$ we have that $(1+\eps)(X-\eps) - X = \eps (X-(1+\eps))\ge 0$. Therefore 
\[e^{-\pi X^2}\ge e^{-4\pi}\cdot e^{-\pi (1+\eps)^2\wt{X}^2} = e^{-4\pi}\cdot e^{-\pi \snorm{n^k(1+\eps)\alpha - (1+\eps)m}_2^2}.\]
Summing over $m\in \Lambda$ and $|n|\le N$, the desired result follows. 
\end{proof}

The key point (which is identical to \cite[Lemma~A.6]{GT09} modulo citing Weyl's inequality for higher degree polynomials) is noting that if $F_{\Lambda,\alpha}(N)$ is small then there exists a vector in $\Lambda^{\ast}$ which is nearly orthogonal to $\alpha$.
\begin{lemma}[Schmidt alternative]\label{lem:scmidt-alter}
There exists $C = C(k)>0$ such that the following holds. Suppose that $\alpha \in \mb{R}^d$ and $\Lambda\subseteq\mb{R}^d$ is a full rank lattice. Let $N$ be a positive integer. Then one of the following always holds:
\begin{itemize}
    \item $F_{\Lambda,\alpha}(N)\ge 1/2$
    \item There exist positive integer $q\ll dA_{\Lambda}^C$ and primitive $\xi\in \Lambda^{\ast}\setminus \{0\}$ such that 
    \[\snorm{\xi}_2\ll \sqrt{d} + \sqrt{\log A_{\Lambda}}\text{ and } \snorm{q\xi \cdot \alpha}_{\mb{R}/\mb{Z}}\ll A_{\Lambda}^{C}N^{-k}.\]
\end{itemize}
($\xi\in\Lambda^\ast$ is \emph{primitive} if $\xi/n\notin \Lambda^{\ast}$ for all $n\ge 2$.)
\end{lemma}
\begin{proof}
We claim that it suffices to prove that if $F_{\Lambda,\alpha}(N)\le 1/2$, then there exists $q\ll A_{\Lambda}^{C}$ and a vector \[\snorm{\xi}_2\ll \sqrt{d} + \sqrt{\log A_{\Lambda}}\text{ and } \snorm{q\xi \cdot \alpha}_{\mb{R}/\mb{Z}}\ll A_{\Lambda}^{C}N^{-k}.\]

Note that the shortest vector in $\Lambda^{\ast}$ is easily seen to be $\gg A_{\Lambda}^{-1}$ by considering the contribution to $A_{\Lambda}$ by scalar multiples of the shortest vector. Therefore $\xi$ is a multiple of $\wt{\xi}$ which is primitive, we have $\snorm{\xi}_2/\snorm{\wt{\xi}}_2\ll (\sqrt{d} + \sqrt{\log A_{\Lambda}})A_{\Lambda}$, and therefore $q$ can be replaced by $q' = q\snorm{\xi}_2/\snorm{\wt{\xi}}_2 \ll (\sqrt{d} + \sqrt{\log A_{\Lambda}})A_{\Lambda}\cdot A_{\Lambda}^{C}\ll d\cdot A_{\Lambda}^{C+1}.$ This gives the desired result. 

Let $M = 4(\sqrt{d} + \sqrt{\log A_{\Lambda}})$ and note that 
\begin{align*}
1/2&\le |F_{\Lambda,\alpha}(N) - 1| = \bigg|\sum_{\xi\in \Lambda^{\ast}\setminus\{0\}}e^{-\pi \snorm{\xi}_2^2}\mb{E}_{|n|\le N}e(n^k\xi\cdot \alpha)\bigg|\\
&\le \sum_{\substack{\xi\in \Lambda^{\ast}\\ \snorm{\xi}_2\ge M}}e^{-\pi \snorm{\xi}_2^2} + \bigg(\sup_{\substack{\xi\in \Lambda^{\ast}\\ \snorm{\xi}_2\le M}}\Big|\mb{E}_{|n|\le N}e(n^k\xi\cdot \alpha)\Big|\bigg)\cdot \sum_{\xi\in \Lambda^{\ast}}e^{-\pi \snorm{\xi}_2^2}\\
&\le e^{-\pi M^2/2}\sum_{\xi\in \Lambda^{\ast}}e^{-\pi \snorm{\xi}_2^2/2} + A_{\Lambda}\sup_{\substack{\xi\in \Lambda^{\ast}\\ \snorm{\xi}_2\le M}}\Big|\mb{E}_{|n|\le N}e(n^k\xi\cdot \alpha)\Big|\\
&= e^{-\pi M^2/2}2^{d/2}\det(\Lambda)\sum_{m\in \Lambda}e^{-2\pi \snorm{m}_2^2} + A_{\Lambda}\sup_{\substack{\xi\in \Lambda^{\ast}\\ \snorm{\xi}_2\le M}}\Big|\mb{E}_{|n|\le N}e(n^k\xi\cdot \alpha)\Big|\\
&\le 1/4 + A_{\Lambda}\sup_{\substack{\xi\in \Lambda^{\ast}\\ \snorm{\xi}_2\le M}}\Big|\mb{E}_{|n|\le N}e(n^k\xi\cdot \alpha)\Big|
\end{align*}
where the equality follows from Poisson summation. We therefore have that 
\[\Big|\mb{E}_{|n|\le N}e(n^k\xi\cdot \alpha)\Big|\ge 1/(4A_{\Lambda})\]
and the result follows from \cref{prop:weyl} as desired. 
\end{proof}

The following appears as \cite[Lemma~A.7]{GT09}.
\begin{lemma}\label{lem:descent}
Let $\Lambda'\subseteq\mb{R}^{d-1}$ and $\Lambda\subseteq\mb{R}^d$ be full-rank lattices with $\Lambda'\subseteq \Lambda$, embedding $\mb{R}^{d-1}$ in $\mb{R}^d$ by putting $0$ in the final coordinate. Suppose that $\alpha\in\mb{R}^d$ and $\alpha'\in\mb{R}^{d-1}$ satisfy $\alpha-\alpha'\in \Lambda$. Then 
\[F_{\Lambda,\alpha}(N)\ge \frac{\det(\Lambda)}{\det(\Lambda')}F_{\Lambda',\alpha'}(N).\]
\end{lemma}

We now combine \cref{lem:scmidt-alter} and \cref{lem:descent} exactly as in \cite[Proposition~A.8]{GT09}.
\begin{lemma}\label{lem:lower-bound}
There exists $C = C(k)>0$ such that the following holds. Suppose $\alpha\in \mb{R}^d$ and $\Lambda\subset \mb{R}^d$ is full rank. If $N>(dA_{\Lambda})^C$ and $F_{\Lambda,\alpha}(N)\le 1/2$, then there exist $\Lambda'\subseteq \mb{R}^{d-1}$, $\alpha'\in \mb{R}^{d-1}$, and $N'\gg N(dA_{\Lambda})^{-C}$ such that
\begin{align*}
A_{\Lambda'} &\ll (\sqrt{d} + \sqrt{\log A_{\Lambda}})A_{\Lambda},\\
F_{\Lambda,\alpha}(N)&\ge(dA_{\Lambda})^{-C}F_{\Lambda',\alpha'}(N^{\ast}).
\end{align*}
\end{lemma}
\begin{proof}
By \cref{lem:scmidt-alter} we have that there exist primitive $\xi\in \Lambda^{\ast}\setminus\{0\}$ and $q\ll dA_{\Lambda}^{C}$ such that 
\[\snorm{\xi}_2\ll \sqrt{d} + \sqrt{\log A_{\Lambda}}\text{ and } \snorm{q\xi \cdot \alpha}_{\mb{R}/\mb{Z}}\ll A_{\Lambda}^{C}N^{-k}.\]

By applying a rotation to both $\alpha$ and $\Lambda$, we may assume that $\xi = \xi_de_d$. Note that $|\xi_d|\gg A_{\Lambda}^{-1}$ and $\snorm{\xi\cdot q^k\alpha}_{\mb{R}/\mb{Z}}\ll d A_{\Lambda}^{O(1)}N^{-k}$. Thus there exists $\beta\in \mb{R}^d$ such that $\beta\cdot \xi \in \mb{Z}$ and 
\[\snorm{\beta-q^k\alpha}_{2}\le |\xi_d|^{-1}\snorm{\xi\cdot q^k\alpha}_{\mb{R}/\mb{Z}}\ll d A_{\Lambda}^{O(1)}N^{-k}.\]

We take $N_{\ast} \gg  (dA_{\Lambda})^{-O(1)}N$ such that $\snorm{\beta-q^k\alpha}_{2}\le N_{\ast}^{-k}/d$. We have that 
\begin{align*}
F_{\Lambda,\alpha}(N)&\gg (dA_{\Lambda})^{-O(1)}F_{\Lambda,\alpha}(N^{\ast})\gg (dA_{\Lambda})^{-O(1)}F_{\Lambda,q^k\alpha}(N^{\ast}/q)\\
&\gg (dA_{\Lambda})^{-O(1)}F_{(1+1/d)\Lambda,(1+1/d)\beta}(N^{\ast}/q);
\end{align*}
where we have applied the first, second, and then third item of \cref{lem:stability}. Let $\pi\colon(x_1,\ldots,x_d)\to(x_1,\ldots,x_{d-1})$ and take $\alpha' = (1+1/d)\pi(\beta)$, $\Lambda' = \pi((1+1/d)\Lambda)$, and $N' = N^{\ast}/q$. (That $\Lambda'$ is a lattice uses that $\xi$ is primitive.) Finally by \cref{lem:descent},
\[F_{(1+1/d)\Lambda,(1+1/d)\beta}(N^{\ast}/q)\ge \frac{\det((1+1/d)\Lambda)}{\det(\Lambda')}\cdot F_{\Lambda',\alpha'}(N') = ((1+1/d)\xi_d)F_{\Lambda',\alpha'}(N')\]
and using the lower bound on $|\xi_d|$ we have the desired lower bound on $F_{\Lambda,\alpha}(N)$. For the upper bound of $A_{\Lambda'}$, by positivity of the exponential function we have 
\[\frac{A_{\pi(\Lambda)}}{\det(\pi(\Lambda))}\le \frac{A_{\Lambda}}{\det(\Lambda)}\]
and therefore 
\[A_{\Lambda'}\le (1+1/d)^{d}A_{\pi(\Lambda)}\le eA_{\Lambda}\frac{{\det(\pi(\Lambda))}}{\det(\Lambda)} = eA_{\Lambda}\cdot \snorm{\xi}_2\]
as desired.
\end{proof}

Iterating \cref{lem:lower-bound} exactly as is done in \cite[Proposition~A.9]{GT09} (see arXiv version for a corrected statement) yields the following.
\begin{proposition}\label{prop:lower-bound}
There exists $C = C(k)>0$ such that the following holds. Let $\Lambda\subseteq \mb{R}^d$ be a lattice of full rank with $\det(\Lambda)\ge 1$. Then for any integer $N$ we have
\[F_{\Lambda,\alpha}(N)\gg d^{-Cd^2}A_{\Lambda}^{-Cd}.\]
\end{proposition}

We now prove \cref{prop:schmidt}; our deduction once again is identical to \cite[Proposition~A.2]{GT09}.
\begin{proof}[Proof of \cref{prop:schmidt}]
Let $R$ be a quantity to be chosen later; $R\gg d^3$ throughout. Let $\Lambda := R\mb{Z}^d$ and $A_{\Lambda} = R^d(\sum_{m\in R\mb{Z}}e^{-\pi m^2})^d\le 2^{O(d)}R^d$. This along with \cref{prop:lower-bound} implies that 
\[F_{\Lambda,\alpha}(N)\gg R^{-O(d^2)}\]
or
\[\mb{E}_{|n|\le N}\sum_{m\in R\mb{Z}^d}e^{-\pi\snorm{n^k\alpha - m}_2^2}\gg R^{-O(d^2)}.\]
For $N\gg (2R)^{\Omega(d^2)}$, there is $n\in \{1,\ldots N\}$ such that 
\[\sum_{m\in R\mb{Z}^d}e^{-\pi\snorm{n^k\alpha - m}_2^2}\gg R^{-O(d^2)}.\]
If $\snorm{n^k\alpha - m}_2\ge \sqrt{R}$ for all $m\in R\mb{Z}^d$, we have 
\begin{align*}
\sum_{m\in R\mb{Z}^d}e^{-\pi\snorm{n^k\alpha - m}_2^2}&\le e^{-\pi R/2}\sum_{m\in R\mb{Z}^d}e^{-\pi\snorm{n^k\alpha - m}_2^2/2}\\
&= e^{-\pi R/2} \frac{2^{d/2}}{\det(\Lambda)}\sum_{\xi\in \Lambda^{\ast}}e^{-2\pi \snorm{\xi}_2^2}\exp(\xi\cdot n^k\alpha)\\
&\le e^{-\pi R/2} \frac{2^{d/2}A_{\Lambda}}{\det(\Lambda)} \ll e^{-\pi R/2} 2^{O(d)}.
\end{align*}
This is a contradiction; thus if $N\ge (2R)^{\Omega(d^2)}$ and $R\gg d^3$ then there exists $1\le n\le N$ such that 
\[\snorm{n^k\alpha_j}_{\mb{R}/\mb{Z}}\le 1/\sqrt{R}\]
for all $j=1,\ldots, d$. The result follows upon taking $R = N^{c/d^2}$ for an appropriately small constant $c$ (if $N$ is small enough that $R\gg d^3$ fails to hold, the result is trivial).
\end{proof}

\section{Constructing an periodic nilsequence with integral structure constants}\label{sec:periodic}
For technical reasons, it will be advantageous to construct nilsequences where the underlying nilmanifold has structure constants which are integral (and in fact divisible by a sufficient fixed integer). This follows via a straightforward lifting procedure where one replaces the underlying Mal'cev basis $\{e_1,\ldots,e_{\dim(G)}\}$ for $\Gamma$ with $\{e_1^{L},\ldots,e_{\dim(G)}^{L}\}$ for a sufficiently large constant $L$. This is primarily to avoid needing various fractional part identities in the proof of \cref{prop:3-step-conversion}.

\begin{lemma}\label{lem:nil-integral}
Fix an integer $K\ge 1$. Suppose we are given a degree $s$ nilmanifold $G/\Gamma$ with dimension $d$ and complexity $M$ and $F\colon G/\Gamma\to \mb{C}$ with Lipschitz norm bounded by $L$ (with respect to the metric $d_{G/\Gamma}$).

There exist $\wt{\Gamma},\wt{F}$ such that $\wt{F}\colon G/\wt{\Gamma}\to\mb{C}$ such that for any $g\in G$ we have  
\[F(g\Gamma) = \wt{F}(g\wt{\Gamma}).\] 
Furthermore we have that $\wt{\Gamma}$ has a Mal'cev basis $\wt{X}$ with all Lie bracket structure constants (the values $c_{ijk}$ in \cref{def:nilmanifold}) being integral and divisible by $K$, that $G/\wt{\Gamma}$ (with $\wt{X}$) has complexity bounded by $O_s(K)\cdot\exp(O(M))$, and that $F$ has Lipschitz norm bounded by $L\cdot (K\cdot \exp(O(M)))^{O_{s}(d^{O_s(1)})}$.
\end{lemma}
\begin{proof}
Let $\mc{X} = \{X_1,\ldots,X_d\}$ denote the Mal'cev basis for the Lie algebra $\mf{g}$ implicit in the complexity bound given for $G/\Gamma$. In order for $\mc{X}$ to be a Mal'cev basis we have:
\begin{itemize}
    \item For each $j = 0,\ldots,d$, the subspace $\mf{h}_j := \on{span}(X_{j+1},\ldots, X_d)$ is a Lie algebra ideal in $\mf{g}$, and hence $H_j := \exp(\mf{h}_j)$ is a normal Lie subgroup of $G$;
    \item If $d_i = \dim(G_i)$ then $G_i = H_{d-d_i}$;
    \item Each $g\in G$ may be written uniquely as $\exp(t_1X_1)\cdots\exp(t_mX_m)$ for $t_i\in \mb{R}$;
    \item $\Gamma$ are exactly the elements which can be written in the form $\exp(t_1X_1)\cdots\exp(t_mX_m)$ with $t_i\in \mb{Z}$ (and is a discrete cocompact subgroup).
\end{itemize}
We also require that the Lie algebra relations are appropriately filtered, corresponding to the containments $[G_i,G_j]\subseteq G_{i+j}$.

We take $\wt{\mc{X}} = \{R\cdot X_1,\ldots,R\cdot X_d\} =: \{\wt{X_1},\ldots,\wt{X_d}\}$ for a large positive integer $R$ to be chosen later. We take $\wt{\Gamma} = \sang{\exp(RX_i)}$ and claim that $\wt{\mc{X}}$ is a valid Mal'cev basis for $G,\wt{\Gamma}$. All conditions for verifying a Mal'cev basis are trivial for us except for the fourth bullet point. The key point is verifying that every element of the subgroup generated by elements $\exp(RX_i)$ can be presented in the desired form.

We take $R = C_s \cdot \on{lcm}(1,\ldots,M)\cdot K$ for some appropriate positive integer $C_s$. Note that 
\[[\wt{X}_i,\wt{X}_j] = \sum_{i,j<k\le m}c_{ijk}\cdot R^2X_k = \sum_{i,j<k\le m}(Rc_{ijk})\wt{X_k}.\]
The crucial point here is that $Rc_{ijk}\in C_s\mb{Z}$ by construction. Let $e_i = \exp(X_i)$ and suppose we have a word in $\wt{\Gamma}$ given by
\[w = e_{i_1}^{\pm R}e_{i_2}^{\pm R}\cdots e_{i_t}^{\pm R}\]
where each $\pm$ denotes an appropriate sign. Note $e_i^R=\exp(\wt{X}_i)$. We first use the Baker--Campbell--Hausdorff formula to write
\[w=\exp\bigg(\sum_{j=1}^dw_j\wt{X}_j\bigg).\]
Note that $w_j\in\mb{Z}$ since the Lie bracket structure constants $Rc_{ijk}$ are in $C_s\mb{Z}$, and Baker--Campbell--Hausdorff only goes to finite height due to the bounded step of the nilpotent Lie group. Now we iteratively pull out a single Mal'cev basis element ``one at a time''. We have
\[\exp(-w_1\wt{X}_1)w=\exp(-w_1\wt{X}_1)\exp\bigg(\sum_{j=1}^dw_j\wt{X}_j\bigg)=\exp\bigg(\sum_{j=2}^dw_j'\wt{X}_j\bigg)\]
for some $w_j'\in\mb{Z}$ using Baker--Campbell--Hausdorff again. Note that we have eliminated use of $\wt{X}_1$ in the right side (which uses the observation that $c_{ijk}=0$ if $k\le\max(i,j)$). We iterate this, obtaining
\[\bigg(\prod_{j=d}^1\exp(-w_j^\ast\wt{X}_j)\bigg)w=\exp(0)=\mr{id}_G\]
for appropriate $w_j^\ast\in\mb{Z}$ (note that the product has decreasing indices left-to-right). Solving for $w$ gives the result. This completes the proof that $\wt{\mc{X}}$ is a valid Mal'cev basis.

Now we define $\wt{F}$ and verify various claims regarding bounds. Note that $G/\wt{\Gamma}$ clearly has complexity bounded by $O(RM)$ which is as desired. Since $\wt{\Gamma}\subseteq \Gamma$, we may lift $F$ from $G/\Gamma$ to $\wt{F}$ on $G/\wt{\Gamma}$ in the obvious manner, composing with a quotient. It suffices to verify that this does not ruin the Lipschitz norm. Let $M' = \exp(O(M))\cdot O_s(K)$; we have that the diameter of $G/\wt{\Gamma}$ is bounded by $(M')^{O_s(d^{O_s(1)})}$ (cf.~\cite[Lemma~A.16]{GT12} with explicit dimension dependence \cite[Lemma~B.8]{Len23b}). Since $\snorm{\wt{F}}_\infty=\snorm{F}_\infty\le\snorm{F}_{\mr{Lip}}\le L$, it suffices to consider points which are within $d_{G/\wt{\Gamma}}$ distance $\eps$ when verifying the Lipschitz bound for $\wt{F}$ as long as we are willing to lose a factor of $\eps^{-1}$.

Note that $d_{G,\wt{\mc{X}}}(x,y)=(1/R)d_{G,\mc{X}}(x,y)$ since Mal'cev coordinates of the second kind in $\mc{X},\wt{\mc{X}}$ differ by a factor of $R$. Now consider $x,y\in G/\wt{\Gamma}$ such that $d_{G/\wt{\Gamma}}(x,y)\le(M')^{-O_s(d^{O_s(1)})}$. There exist representatives of $\wt{x}$ and $\wt{y}$ for $x$ and $y$ with respect to $\wt{\Gamma}$ such that 
\[d_{G/\wt{\Gamma}}(x,y) = d_{G,\wt{\mc{X}}}(\wt{x},\wt{y})\text{ and }d_{G,\wt{\mc{X}}}(\wt{x},\mr{Id}_G) + d_{G,\wt{X}}(\wt{y},\mr{Id}_G)\le(M')^{O_s(d^{O_s(1)})}.\]
Additionally, we claim that since $x,y$ are sufficiently close in $G/\wt{\Gamma}$, we find that
\[d_{G/\Gamma}(\wt{x}\Gamma,\wt{y}\Gamma)=d_{G,\mc{X}}(\wt{x},\wt{y}).\]
The argument is as follows. If this is not true, then there is $\gamma\in\Gamma\setminus\{0\}$ with $d_{G,\mc{X}}(\wt{x},\wt{y}\gamma)<d_{G,\mc{X}}(\wt{x},\wt{y})$. By approximate left-invariance of the metric on $G$ (\cite[Lemma~A.5]{GT12} with explicit dimension dependence \cite[Lemma~B.4]{Len23b}) we see
\[d_{G,\mc{X}}(\wt{y}^{-1}\wt{x},\gamma)\le(M')^{O_s(d^{O_s(1)})}Rd_{G/\wt{\Gamma}}(x,y)\le(M')^{O_s(d^{O_s(1)})}\eps.\]
Additionally,
\[d_{G,\mc{X}}(\wt{y}^{-1}\wt{x},\mr{id}_G)\le(M')^{O_s(d^{O_s(1)})}\eps.\]
The triangle inequality yields
\[d_{G,\mc{X}}(\gamma,\mr{id}_G)\le(M')^{O_s(d^{O_s(1)})}\eps.\]
Now, the metric is comparable to the $L^\infty$ distance in second kind Mal'cev coordinates up to a factor of $(M')^{O_s(d^{O_s(1)})}$ (\cite[Lemma~A.4]{GT12} with explicit dimension dependence \cite[Lemma~B.3]{Len23b}) as long as $\eps$ is sufficiently small here, so we find
\[1\le(M')^{O_s(d^{O_s(1)})}\eps.\]
This as a contradiction as long as we take sufficiently small $\eps$ (which will be admissible in the bounds we need). Finally,
\[|\wt{F}(x\wt{\Gamma})-\wt{F}(y\wt{\Gamma})|=|F(\wt{x}\Gamma)-F(\wt{y}\Gamma)|\le Ld_{G/\Gamma}(\wt{x},\wt{y})=Ld_{G,\mc{X}}(\wt{x},\wt{y})=LRd_{G,\wt{\mc{X}}}(\wt{x},\wt{y})=LRd_{G/\wt{\Gamma}}(x,y)\]
and we are done.
\end{proof}

Having constructed a nilsequence with integral Lie bracket structure constants, it will also prove useful to construct a nilsequence which is additionally periodic. For this purpose, we quantify a construction of Manners \cite{Man14} which demonstrates that one may lift a nilsequence along a subset of the support.  We first give a quantified version of \cite[Theorem~1.5]{Man14}.

\begin{proposition}\label{prop:phi-version}
There is an integer $C_s\ge 1$ so that the following holds. Fix an integer $K\ge 1$. Suppose we have a degree $s$ nilmanifold $G/\Gamma$ with complexity $M$ and $F\colon G/\Gamma\to\mb{C}$ with Lipschitz norm $L$. Furthermore suppose we have a polynomial sequence $g\colon\mb{Z}\to G$ and a smooth function $\phi\colon\mb{R}/\mb{Z}\to[0,1]$ with $\on{supp}(\phi)\in [(3K)^{-1},(2K)^{-1}]$. 

There exists a degree $s$ nilmanifold $\wt{G}/\wt{\Gamma}$ with a polynomial sequence $\tilde{g}\colon\mb{Z}\to \wt{G}$ such that 
\[\wt{F}(\wt{g}(x)) = \phi(x/N) F(g(x \imod N)),\]
where $x \imod N$ is interpreted to lie in $\{0,\ldots,N-1\}$. Furthermore we may assume:
\begin{itemize}
    \item $\wt{G}/\wt{\Gamma}$ has complexity bounded by $O_K(1)^{O_s(1)} + O_s(M)$ and dimension at most $O_s(d)$;
    \item $\wt{F}$ has Lipschitz norm bounded by $L\cdot O_{\phi,K}(M^{O_s(d^{O_s(1)})})$;
    \item $\wt{g}(n)^{-1}\wt{g}(n+N)\in\wt{\Gamma}$ for all $n\in\mb{Z}$;
    \item If all Lie bracket structure constants of $G/\Gamma$ are integers divisible by $K$ then the Lie bracket structure constants of $\wt{G}/\wt{\Gamma}$ are contained in $K(C_s)^{-1}\mb{Z}$.
\end{itemize}
\end{proposition}

Given this we deduce a variant of the $U^4$-inverse theorem from \cref{thm:U4-new}. This argument is essentially a combination of \cref{lem:nil-integral}, \cref{prop:phi-version}, and the argument of \cite[Theorem~1.4]{Man14}.

\begin{theorem}\label{thm:U4-modified}
Suppose we have an integer $K\ge 1$. There exists a constant $C=C_K>0$ such that the following holds. Let $N$ be prime, let $\delta>0$, and suppose that $f\colon\mb{Z}/N\mb{Z}\to\mb{C}$ is a $1$-bounded function with 
\[\snorm{f}_{U^4(\mb{Z}/N\mb{Z})}\ge\delta.\]
There exists a degree $3$ nilsequence $F(g(n)\Gamma)$ such that it has dimension bounded by $(\log(1/\delta))^{C}$, complexity bounded by $O_K(1)$, such that $F$ is $1$-Lipschitz, and such that
\[\big|\mb{E}_{n\in[N]} f(n)\ol{F(g(n)\Gamma)}\big|\ge1/\exp((\log(1/\delta))^{C}).\]
Furthermore, all Lie bracket structure constants of $G$ (the $c_{ijk}$ in \cref{def:nilmanifold}) are integers divisible by $K$, $g(n)g(n+N)^{-1}\in\Gamma$ for all $n\in\mb{Z}$, and $g(0) = \on{id}_G$.
\end{theorem}
\begin{remark}
In all our applications, we will take $K = 12$. 
\end{remark}

We deduce \cref{thm:U4-modified} from \cref{prop:phi-version}; this is essentially as in the proof of \cite[Theorem~1.4]{Man14}.

\begin{proof}[Proof of \cref{thm:U4-modified}]
By adjusting constants appropriately, we may assume that $N$ is larger than an absolute constant (and for small $N$ apply the $U^2$-inverse theorem noting all Lie bracket structure constants for an abelian nilmanifold are $0$).

We take a partition of unity of $\mb{R}/\mb{Z}$ denoted $\phi_1,\ldots,\phi_{20K}\colon\mb{R}/\mb{Z}\to\mb{R}$ such that $\on{supp}(\phi_j)\subseteq[j/(20K),j/(20 K) + 1/(10K)]\imod 1$. We abusively denote $\phi_i\colon\mb{Z}/N\mb{Z}\to\mb{R}$ via $\phi_i(n) := \phi_i(n/N)$. Recalling that $\snorm{\cdot}_{U^4(\mb{Z}/N\mb{Z})}$ is a norm, we have 
\[\snorm{f}_{U^4(\mb{Z}/N\mb{Z})}\le \sum_{i=1}^{20K}\snorm{\phi_i f}_{U^4(\mb{Z}/N\mb{Z})}.\]
Thus there exists an index $k$ such that $\snorm{\phi_k f}_{U^k(\mb{Z}/N\mb{Z})}\ge \delta/(20 K)$. Applying a translation, we may assume that $\on{supp}(\phi_k f)$ is contained in $[\lceil N/(3K)\rceil,\lfloor N/(2K)\rfloor] \imod N$. 

Applying \cref{thm:U4-new}, we have that there exists degree $3$ nilsequence $F_1(g_1(n)\Gamma_1)$ on $G/\Gamma_1$ with complexity $O(1)$, dimension $(\log(1/\delta))^{O(1)}$, and Lipschitz constant $1$ such that 
\[|\mb{E}_{n\in [N]}\phi_k(n/N) f(n) \ol{F_1(g_1(n)\Gamma)}|\ge 1/\exp((\log(1/\delta))^{O(1)}).\]
By \cref{lem:nil-integral}, we may construct $\Gamma_2,F_2$ with Lie bracket structure constants for $G/\Gamma_2$ being integers divisible by $K\cdot C_s$ (where $s=3$) such that 
\[|\mb{E}_{n\in [N]}\phi_k(n/N) f(n) \ol{F_2(g(n)\Gamma_2)}|\ge 1/\exp((\log(1/\delta))^{O(1)}).\]
Furthermore $G/\Gamma_2$ has complexity bounded by $O(K)$ and $F_2$ has Lipschitz norm bounded by $(2K)^{O(d^{O(1)})}$. Next, we may write
\[\mb{E}_{n\in [N]}\phi_k(n/N) f(n) \ol{F_2(g_2(n)\Gamma_2)} = \mb{E}_{n\in [N]} f(n) \ol{\phi_k(n/N) F_2(g_2(n)\Gamma_2)}\]
and apply \cref{prop:phi-version} to $\phi(x/N) F_2(g_2(x)\Gamma_2)$. In particular, we obtain $F_3$, $G_3$, $\Gamma_3$, and $g_3(n)$ such that 
\[\mb{E}_{n\in [N]} f(n) \ol{\phi_k(n/N) F_2(g_2(n)\Gamma_2)} = \mb{E}_{n\in\mb{Z}/N\mb{Z}} f(n) \ol{F_3(g_3(n)\Gamma)}\]
with $F_3$ having Lipschitz norm $(O_K(1))^{O(d^{O(1)})}$, $G_3/\Gamma_3$ having complexity $O_K(1)$, $g_3(n)$ being $N$-periodic, and all Lie structure constants divisible by $K$. 

Note that we do not necessarily have $g_3(0) \neq \on{id}_{G_3}$. This may be repaired by defining $g(n) = \{g_3(0)\}^{-1}g_3(n)g_3(0)^{-1}\{g_3(0)\}$ where $\{g_3(0)\}^{-1}g_3(0) = [g_3(0)]\in \Gamma_3$ and $\{g_3(0)\}$ has Mal'cev coordinates of the second kind bounded by $1/2$. Taking 
\[F(x\Gamma_3) := F_3(\{g_3(0)\}x\Gamma_3)\]
we have
\[\mb{E}_{n\in\mb{Z}/N\mb{Z}} f(n) \ol{F_3(g_3(n)\Gamma_3)} = \mb{E}_{n\in\mb{Z}/N\mb{Z}} f(n) \ol{F(g(n)\Gamma_3)}.\]
As left-multiplication by bounded elements preserves the metric up to an admissible multiplicative factor (\cite[Lemma~B.4]{Len23b}), we find that $F$ has Lipschitz norm $(O_K(1))^{O(d^{O(1)})}$. Thus, rescaling $F$ to have Lipschitz norm $1$ will keep the correlation sufficiently large. Furthermore as left-multiplying a fixed element and right-multiplying a fixed element in $\Gamma_3$ does not change whether our polynomial sequence is $N$-periodic on $G_3/\Gamma_3$ this completes the proof. 
\end{proof}

We now prove \cref{prop:phi-version}. We use the construction of Manners \cite[Theorem~1.5]{Man14}; we modify the construction slightly to match the definition of filtration used in this paper.

\begin{proof}[Proof of \cref{prop:phi-version}]
\textbf{Step 1: Setup and constructing the proposed Mal'cev basis.}
We consider the given degree $s$ nilpotent Lie group $G$ with filtration 
\[G=G_0 = G_1\geqslant G_2\geqslant\cdots\geqslant G_s\geqslant\on{Id}_G.\]
Let $\mf{g}_j=\log_G(G_j)$.

For $i\ge 0$, let $G^{+i}$ denote the prefiltration $G_i\geqslant G_{i+1}\geqslant G_{i+2}\geqslant\cdots \geqslant G_s\geqslant\on{Id}_G$. Let $H_i = \on{poly}(\mb{Z},G^{+i})$, where we define polynomial sequences with respect to prefiltrations as in \cite[Definition~B.1]{GTZ12} with the prefiltration $\mb{Z}\geqslant\mb{Z}\geqslant\{0\}\geqslant\cdots$ on the domain $\mb{Z}$ (see also \cite[p.~28]{GT12}, which includes the formal definition of prefiltrations).

By the Filtered Lazard--Leibman theorem of Green, Tao, and Ziegler \cite[Proposition~B.6]{GTZ12}, $H_i$ are not only groups but also can be given the prefiltration
\begin{align*}
H_0&\geqslant H_1\geqslant H_2\geqslant\cdots\geqslant H_s\geqslant\on{Id}_{H_0}.
\end{align*}
This implies that one has the genuine filtration 
\[H_1=H_1\geqslant H_2\geqslant\cdots\geqslant H_s\geqslant\on{Id}_{H_0}.\]
We now define the semidirect product $H_i\rtimes\mb{R}$ by defining a group operation $\ast$ via
\[(x\mapsto g(x),t)\ast (x\mapsto g'(x),t') = (x\mapsto g(x + t')g'(x),t+t').\]
Note that this is slightly different than that given in the work of Manners \cite{Man14}, since we take right-quotients by the cocompact subgroup, and in fact this is rather the opposite group of the semidirect product of the opposite groups (but we suppress such notational dependence). The semidirect product in question is embedding $t\in\mb{R}$ as the ``shift by $t$'' automorphism. Additionally, we abusively identify $g\in\on{poly}(\mb{Z},G^{+i})$ with a function $\mb{R}\to G^{+i}$ defined by using Taylor expansion and then allowing the argument to vary over reals instead of integers using the Lie group structure; this extension is unique due to a generalization of the identity theorem.

One can easily prove that 
\[H_0\rtimes\mb{R}\geqslant H_1\rtimes\mb{R}\geqslant H_2\rtimes\{0\}\geqslant\cdots\geqslant H_s\rtimes\{0\}\geqslant\on{Id}_{H_0\rtimes\mb{R}}\]
is a prefiltration; this follows immediately from the proof given in \cite[Proposition~14]{HK18}. (The only nontrivial aspect involves the $\rtimes\mb{R}$ component, and it is not too hard to show that it suffices to check relevant properties at the level of the generators $(\on{id}_{H_0},t)$ since we have the above prefiltration on $H_0$.)

Let $X_1,\ldots,X_d$ denote the Mal'cev basis on the Lie algebra associated to $G$ (which is adapted to the $s$-step filtration given at the beginning). We construct the desired Mal'cev basis for $H_0\rtimes\mb{R}$ as follows:
\begin{itemize}
    \item Let $Z = (0,K)$.
    \item Let $Y_{i, k} = (n \mapsto \binom{n}{k}X_i, 0)$; if $X_i\in \mf{g}_j\setminus \mf{g}_{j+1}$ define the \emph{level} of $Y_{i,k}$ to be $j-k$. We restrict attention to $Y_{i,k}$ of nonnegative level and level at most $s$. (Note that there are no contributions with $j=0$.)
    \item We order the Mal'cev basis left-to-right by increasing level; within level we order with increasing $k$, and then in increasing order of $i$. $Z$ is inserted as the first element of level $1$.
\end{itemize}
\begin{remark}
By Taylor expansion (see \cite[Lemma~B.9]{GTZ12}) we see that this is a Lie basis of the necessary group, filtered with respect to the desired prefiltration. Note that \emph{a priori} the Lie algebra of $\on{poly}(\mb{Z},G^{+0})$ is some abstract space, but we can utilize the embedding $H_0=\on{poly}(\mb{Z},G^{+0})\hookrightarrow G^{\mb{Z}}$ and the corresponding representation of Lie algebras to yield the above form of writing elements of the tangent space.
\end{remark}

Our goal is now to verify that the discrete cocompact subgroup $\wt{\Gamma} = \on{poly}(\mb{Z},\Gamma)\rtimes (K\mb{Z})$ is the push-forward of $\mb{Z}^{\dim(\wt{G})}$ for Mal'cev coordinates of the second kind. This algorithmic proof will also immediately verify the remaining property of the Mal'cev basis (regarding rationality of the Lie bracket structure constants). After this, we will give the necessary nilsequence lifting construction and verify the necessary properties.

\noindent\textbf{Step 2: Verifying Mal'cev properties modulo the semidirect product.}
We first prove that the $Y_{i,k}$ graded as above form a Mal'cev basis for 
\[H_0\geqslant H_1\geqslant H_2\geqslant\cdots\geqslant H_s\geqslant\on{Id}_{H_0}\]
with discrete cocompact subgroup $\Gamma'=\on{poly}(\mb{Z},\Gamma)$. We proceed by induction upwards by step. Note that handling the final step is trivial as the group $H_s$ is abelian.

Assume the inductive hypothesis that $H_{j+1}\cap\Gamma'$ is the image of the integer lattice in Mal'cev coordinates of the second kind (which necessarily uses only those basis elements of level at least $j+1$). Using Taylor expansion (see \cite[Lemma~B.9]{GTZ12}), we may write an element $H_{j}\cap\Gamma'$ as 
\begin{equation}\label{eq:taylor-Hj}
n\mapsto \prod_{i=0}^{s-j}g_{i+j}^{\binom{n}{i}}
\end{equation}
where $g_i\in G_{i+j}\cap \Gamma$. ($g_i\in\Gamma$ easily follows from our polynomial lying in $\Gamma'=\on{poly}(\mb{Z},\Gamma)$ sequence.)

Now, for our given value of $j$, we prove by an induction on $\ell$ that the set of polynomials of the form 
\[n\mapsto \prod_{i=0}^{s-j}g_{i+j}^{\binom{n}{i}}\]
with $g_{i+j}\in G_{i+j+1}\cap \Gamma$ for $0\le i\le\ell$ and $g_i\in G_{i+j}\cap\Gamma$ for $\ell+1\le i\le s-j$ is generated by the level $j+1$ elements and the largest $s-j-\ell$ ``types'' of level $j$ elements of the Mal'cev basis for $H_0$. (Here ``types'' means in the sense of the parameter $k$ we used to define the ordering on the $Y_{i,k}$ above.) The case $\ell = s-j$ is exactly the above inductive assumption regarding $H_{j+1}\cap\Gamma'$, and the case $\ell=0$ is exactly what we need to complete the (outer) induction.

We now suppose that this is known for some $s-j\ge\ell\ge 1$, and wish to prove the result for $\ell-1$. We may write
\[\bigg(n\mapsto \prod_{i=0}^{s-j}g_{i+j}^{\binom{n}{i}}\bigg) = \bigg(n\mapsto \prod_{i=0}^{\ell-1}g_{i+j}^{\binom{n}{i}}\bigg) \bigg(n\mapsto g_{\ell + j}^{\binom{n}{\ell}}\bigg)\bigg(n\mapsto \prod_{i=\ell + 1}^{s-j}g_{i+j}^{\binom{n}{i}}\bigg)=ABC,\]
where $A,B,C$ are defined in the obvious way. Note that $A\in H_{j+1}\cap\Gamma'$. We may write $g_{\ell+j} = \prod_{r}\exp(t_rX_r)$ for $t_r\in\mb{Z}$ and $X_r\in \mf{g}_{\ell+j}$. Note that 
\[D^{-1}B=\bigg(\prod_{r}\bigg(n\mapsto \exp(t_rX_r)^{\binom{n}{\ell}}\bigg)\bigg)^{-1}\bigg(n\mapsto g_{\ell + j}^{\binom{n}{\ell}}\bigg) \in H_j\cap\Gamma'\]
where $D$ is defined in the obvious way. The Taylor expansion of the polynomial $(D^{-1}B)C$ (in a similar form to \cref{eq:taylor-Hj}) trivially has its coordinates of index $0,\ldots,\ell-1$ being the identity (plugging in the values $n=0,1,\ldots,\ell-1$). The $\ell$th Taylor coordinate is in $G_{\ell + j + 1}\cap\Gamma'$, considering the value $n=\ell$. Thus
\[ABC=D(D^{-1}AD)(D^{-1}BC),\]
where $D^{-1}BC$ has the form
\[n\mapsto\prod_{i=\ell}^{s-j}g_{i+j}'^{\binom{n}{i}}\]
for $g_{i+j}'\in G_{i+j}\cap\Gamma$ and additionally $g_{\ell+j}'\in G_{\ell+j+1}\cap\Gamma$.

Note that $D^{-1}BC$ is of the correct form of Taylor series to apply the inductive hypothesis (since $g_{\ell+j}'\in G_{\ell+j+1}$), yielding
\[D^{-1}BC=\prod_{\mu\in M}\exp(t_\mu X_\mu)\cdot F=EF\]
where $M$ ranges over the $s-j-\ell$ ``types'' of level $j$ elements of the Mal'cev basis of $H_0$, from $\ell+1$ to $s-j$ inclusive, and is arranged in the correct order within the product $E$, and $F\in H_{j+1}\cap\Gamma'$. Finally, we may write
\[ABC=DE(E^{-1}(D^{-1}AD)E\cdot F),\]
and note that $H_{j+1}$ is normal in $H_0$ due to the original prefiltration we established, so we find $E^{-1}(D^{-1}AD)E\in H_{j+1}\cap\Gamma'$. Thus we have
\[ABC=DEF'\]
where $D,E$ form the necessary ordered product for the Mal'cev second kind representation, and $F'\in H_{j+1}\cap\Gamma'$. Finally, we recall the induction hypothesis for $H_{j+1}\cap\Gamma'$ to finish.

\noindent\textbf{Step 3: Semidirect product and Lie bracket structure constants.}
We now handle the generator $Z$. Note that given an element $(p,t)\in\on{poly}(\mb{Z},\Gamma)\rtimes (K\mb{Z})$, we may write
\[(p,t) = (\on{id}_{G},t) \ast (p, 0)\]
and therefore if $Z$ was the highest element of the ordering (which would correspond to a situation where the filtration contains $H_1\rtimes\{0\}$ instead of $H_1\rtimes\mb{R}$) we would easily be done. Instead, though, it is inserted right before the Mal'cev basis components for $H_1$. Note however that
\[(\on{id}_{G},-K)\ast(p^{-1},0)\ast(\on{id}_{G},K)\ast(p,0)=(p^{-1},-K)\ast(p,K)=(x\mapsto p^{-1}(x+K)p(x),0)\]
and recall that derivatives of polynomials in $H_j\cap\on{poly}(\mb{Z},\Gamma)$ are in $H_{j+1}\cap\on{poly}(\mb{Z},\Gamma)$ due to the definitions of the shifted filtrations on $G$. Therefore we may commute $(\on{id}_G,t)$ across the product of the level $0$ terms in the representation above, and we will introduce only terms of level $1$ and higher (and they will lie in $\wt{\Gamma}$). Then using the Mal'cev representation for $H_1\cap\Gamma'$ finishes. This completes our discussion that the generators listed form a Mal'cev basis.

We now compute the Lie bracket structure constants associated to the Mal'cev basis. We will compute all structure constants via the identity 
\begin{align*}
[V,W] &= \frac{d}{ds}\frac{d}{dt} \exp(sV)\exp(tW)\exp(-sV)\exp(-tW)\bigg|_{u,v=0}
\end{align*}
which holds for any Lie group and the associated Lie bracket.

We first compute the constants associated with $Z$. Note that 
\begin{align*}
[Y_{i,k},Z] &= \frac{d}{dt_0}\frac{d}{dt_1}\big(n\mapsto \exp(t_0X_i)^{\binom{n}{k}},0\big)(\on{id}_G,Kt_1)\big(n\mapsto \exp(t_0X_i)^{-\binom{n}{k}},0\big)(\on{id}_G,-Kt_1)\bigg|_{t_0,t_1=0}\\
&=\frac{d}{dt_0}\frac{d}{dt_1}\big(n\mapsto\exp(t_0X_i)^{\binom{n}{k}-\binom{n-Kt_1}{k}},0\big)\bigg|_{t_0,t_1=0}\\
& = \log_{H_0\rtimes\mb{R}}\bigg(n\mapsto\exp(X_i)^{-\frac{d}{dt_1}\binom{n-Kt_1}{k}\big|_{t_0=0}},0\bigg)=\bigg(n\mapsto-\frac{d}{dt_1}\binom{n-Kt_1}{k}X_i\bigg|_{t_1=0},0\bigg),
\end{align*}
where the last line follows from considering the relationship between the differential structure on $H_0$ and on $G$ (which allows us to ``implicitly differentiate in the obvious way''). The coefficients of $\frac{d}{dt_1}\binom{n-Kt_1}{k}\big|_{t_1=0}$ are all integer multiples of $K^t/k!$ for $1\le t\le k-1$. We may easily express this in terms of various $Y_{i,k'}$ for $k'<k$, and the structure constants clearly have the desired form, lying in $K(C_s)^{-1}\mb{Z}$ for an appropriate integer $C_s$.

We next compute the structure constants associated with two polynomials. Notice that 
\begin{align*}
&[Y_{i,k},Y_{j,\ell}]\\
&= \frac{d}{dt_1}\frac{d}{dt_2}\big(n\mapsto \exp(t_1X_i)^{\binom{n}{k}},0\big)\big(n\mapsto \exp(t_2X_j)^{\binom{n}{\ell}},0\big)\big(n\mapsto \exp(t_1X_i)^{-\binom{n}{k}},0\big)\big(n\mapsto \exp(t_2X_j)^{-\binom{n}{\ell}},0\big)\bigg|_{t_1,t_2=0}\\
&=\frac{d}{dt_1}\frac{d}{dt_2}\bigg(n\mapsto \exp\bigg(t_1t_2\binom{n}{k}\binom{n}{\ell}[X_i,X_j]\bigg),0\bigg)\bigg|_{t_1,t_2=0}\\
&=\bigg(n\mapsto\binom{n}{k}\binom{n}{\ell}[X_i,X_j],0\bigg).
\end{align*}
The second equality follows from using Baker--Campbell--Hausdorff to multiply the polynomial sequences (as pointwise functions, say) and collapse them, and then noting that we may discard terms with higher powers than $t_1^1$ and $t_2^1$. The only relevant remaining term is the $t_1t_2$ which has the claimed form.

Note that as $\binom{n}{k}\binom{n}{\ell}$ is an integer-valued polynomial, by Polya's classification of integer polynomials \cite{Pol15}, it follows that this may be written as an integral combination of binomial coefficients with coefficients bounded by $O_s(1)$. We obtain a representation in terms of various $Y_{i',k'}$ where $k'\le k+\ell$ and $X_{i'}$ shows up in the expansion of $[X_i,X_j]$. We are done with this part of the proof, which includes verifying the well-definedness of various lifted filtrations, Mal'cev bases, and also the rationality properties of the Lie bracket structure constants.

\noindent\textbf{Step 4: Lifting the nilsequence and checking quantitative dependences.}
We now actually define the desired function. We are given a polynomial sequence $g(x)$ and we define $q(x) = g(xN/K)$. We define 
\[\wt{g}(x) = (q,0)\ast(\on{id}_G,Kx/N).\]
Note that $\wt{g}(x)$ is $N$-periodic as $\wt{g}(x+N)=\wt{g}(x)\ast(\on{id}_G,K)$. We now define $\wt{F}$ on $\on{poly}(\mb{R},G)/\on{poly}(\mb{Z},\Gamma)\times [0,K)$ via
\[\wt{F}(x\cdot \on{poly}(\mb{Z},\Gamma),t) = \phi(t/K)\cdot F(x(0)\Gamma).\]
This extends uniquely to $H_0\rtimes\mb{R}$ if we enforce right-$\wt{\Gamma}$-invariance, which yields
\[\wt{F}(x\cdot \on{poly}(\mb{Z},\Gamma),t) = \phi(\{t/K\})\cdot F(x(-K\lfloor t/K\rfloor)\Gamma).\]
That $\wt{F}$ is appropriately Lipschitz and that we can then descend this construction to a proper filtration will be checked last.

The key point is the following computation:
\begin{align*}
\wt{F}(\wt{g}(x)) &= \wt{F}((q,0)\ast (\on{id}_G,Kx/N))=\wt{F}((q(\cdot + Kx/N),Kx/N))\\
&=\phi(\{x/N\})F(q(Kx/N - K\lfloor x/N\rfloor))=\phi(\{x/N\})F(g(x - N\lfloor x/N\rfloor))\\
&=\phi(\{x/N\})F(g(x \imod N)).
\end{align*}
We now check that the function $|\wt{F}|$ is appropriately Lipschitz with respect to the Mal'cev basis specified. Since $\wt{F}$ is bounded by $\snorm{F}_{\infty}\snorm{\phi}_{\infty}$, note that for $x,y\in (H_0\rtimes \mb{R})/\wt{\Gamma}$ if $d_{(H_0\rtimes \mb{R})/\wt{\Gamma}}(x,y)\ge \eps$ we have 
\[\frac{\wt{F}(x)-\wt{F}(y)}{d_{(H_0\rtimes \mb{R})/\wt{\Gamma}}(x,y)}\le 2\eps^{-1}\snorm{F}_{\infty}\snorm{\phi}_{\infty}.\]

We now assume that $\eps\le O_K(M)^{-O_s(d)^{O_s(1)}}$ to be chosen later. The diameter of $(H_0\rtimes \mb{R})/(\poly(\mb{Z},\Gamma)\rtimes K\mb{Z})$ is bounded by $O_K(M)^{O_s(d)^{O_s(1)}}$ by \cite[Lemma~B.8]{Len23b}\footnote{We have proven that the constructed basis is a filtered basis for the Lie algebra. Separating by level and then by ``type'', we see that the filtered basis has nesting property of degree $\le 2s^2$. \cite[Lemma~B.8]{Len23b} is proven under only the assumption that the basis for the Lie algebra exhibits this nesting property.}
, so there exist $\wt{x},\wt{y}$ such that 
\[d_{(H_0\rtimes \mb{R})/\wt{\Gamma}}(x,y) = d_{(H_0\rtimes \mb{R})}(\wt{x},\wt{y})\text{ and }d_{G,\wt{\mc{X}}}(\wt{x},\mr{Id}_G) + d_{G,\wt{X}}(\wt{y},\mr{Id}_G)\le O_K(M)^{O_s(d^{O_s(1)})}.\]

Furthermore we may chose the fundamental representatives $\wt{x}$ and $\wt{y}$ such that the second coordinate is in $[0,K)$ (by multiplying by an appropriate element of $\on{Id}_{G}\rtimes(K\mb{Z})$ on the right and noting that the metric is right-invariant). Note that if the second coordinate of the representative $\wt{x}$ is outside the range $[1/4,2/3]$, the assumption on the support of $\phi$ guarantees that $\wt{F}(\wt{x}\Gamma) = \wt{F}(\wt{y}\Gamma) = 0$ and therefore verifying the Lipschitz constant in this case is trivial. (This uses that $\eps$ is sufficiently small.)

Else we denote $\wt{x} = (x^{\ast},x')$ with $x^{\ast}\in H_0$ and $x'\in[1/4,2/3]$ and $\wt{y}$ analogously. We now have 
\begin{align*}
\big|\wt{F}(\wt{x}\wt{\Gamma}) - \wt{F}(\wt{y}\wt{\Gamma})\big| &= \big|\phi(x'/K) F(x^{\ast}(0)\Gamma) - \phi(y'/K) F(y^{\ast}(0)\Gamma)\big| \\
& \le \snorm{\phi}_{\infty} \big|F(x^{\ast}(0)\Gamma) - F(y^{\ast}(0)\Gamma)\big| + \big|\phi(x'/K)-\phi(y'/K)\big|\snorm{F}_{\infty}.
\end{align*}
Let $\psi$ denote the second Mal'cev coordinates with respect to the basis specified for $(H_0\rtimes \mb{R})/\wt{\Gamma}$. Note that $x',y'$ are controlled only by $Z$ in the Mal'cev basis and thus we can bound the second term by
\[\big|\phi(x^{\ast}/K)-\phi(y^{\ast}/K)\big|\lesssim_K \snorm{\psi'}_{\infty}\cdot \big|\psi(\wt{x})-\psi(\wt{y})\big|.\]
To bound the first term, note that 
$d_{G/\Gamma}(x^{\ast}(0)\Gamma,y^{\ast}(0)\Gamma)\le d_G(x^{\ast}(0),y^{\ast}(0))$. 

Let $\wt{x} = \exp((x'/K)Z)\prod_{k,\ell}\exp(x_{k,\ell}Y_{k,\ell})$ and analogously for $\wt{y}$ where the product over $Y_{k,\ell}$ is taken in the ordering specified for the Mal'cev basis. We have 
\[x^{\ast}(0) = \prod_{k,\ell}\exp\bigg(\binom{0}{\ell}x_{k,\ell}X_k\bigg)\text{ and }y^{\ast}(0) = \prod_{k,\ell}\exp\bigg(\binom{0}{\ell}y_{k,\ell}X_k\bigg).\]

Note that the only terms which contribute to the above product are $Y_{k,0} = (n\mapsto \binom{n}{0}X_k) = (n\mapsto X_k)$. If $X_k\in \mf{g}_{j}\setminus \mf{g}_{j+1}$, this has level precisely $j$ and thus the product (removing terms which are the identity) are in the correct order for the Mal'cev basis on $G$.

Therefore we have 
\begin{align*}
d_{G}(x^{\ast}(0),y^{\ast}(0)) &\le O_K(M)^{O_s(d^{O_s(1)})}\max_{k}|x_{k,0}-y_{k,0}|\\
&\le O_K(M)^{O_s(d^{O_s(1)})}\big(\max_{k,\ell}|x_{k,\ell}-y_{k,\ell}| + |x'/K-y'/K|\big)\\
&\le O_K(M)^{O_s(d^{O_s(1)})}d_{H_0\rtimes \mb{R}}(\wt{x},\wt{y}).
\end{align*}
The first line follows from \cite[Lemma~B.3]{Len23b} applied on the Mal'cev basis $\mc{X}$ of $G$ and the second inequality is trivial. The final inequality follows from \cite[Lemma~B.3]{Len23b} and that $\max_{k,\ell}|x_{k,0}-y_{k,0}| + |x'/K-y'/K|$ corresponds (up to constants) to the $L^{\infty}$ distance in the Mal'cev coordinates when $Z$ is placed first in the order and one may return to the original distance at the cost of $O_K(M)^{O_s(d)^{O_s(1)}}$ by \cite[Lemma~B.3]{Len23b}. (Shifting $Z$ to the first position is clearly a $1$-rational change of basis.)

\noindent\textbf{Step 5: Reducing to a proper filtration.}
We now (finally) perform the last reduction to place our polynomial sequence on a proper filtration. By inspection we have
\[\wt{g}(x) = (h_0^{\ast}) (h_1^{\ast})^{\binom{x}{1}}\]
with $h_i^{\ast}\in H_i\rtimes \mb{R}$ for $i\le 1$. Let $h_0^{\ast} = \{h_0^{\ast}\}[h_0^\ast]$ such that $[h_0^\ast]\in\wt{\Gamma}$ and $\snorm{\psi(\{h_0^{\ast}\})}\le 1$. 

For $\wt{p}\in H_1\rtimes\mb{R}$ define $F^{\ast}(\wt{p}\wt{\Gamma}) = \wt{F}(\{h_0^{\ast}\}\wt{p}\wt{\Gamma})$, and note that
\begin{align*}
\wt{F}(\wt{g}(x)\wt{\Gamma}) &= \wt{F}(\{h_0^{\ast}\}[h_0^\ast](h_1^{\ast})^{\binom{x}{1}}\wt{\Gamma})=\wt{F}(\{h_0^{\ast}\}([h_0^\ast]h_1^{\ast}[h_0^\ast]^{-1})^{\binom{x}{1}}[h_0^\ast]\wt{\Gamma})\\
&= F^{\ast}(([h_0^\ast]h_1^{\ast}[h_0^\ast]^{-1})^{\binom{x}{1}}\wt{\Gamma}).
\end{align*}
Since left-multiplication by bounded elements is suitably Lipschitz and since $H_1\rtimes \mb{R}$ is normal in $H_0\rtimes \mb{R}$, we have that $([h_0^\ast]h_1^{\ast}[h_0^\ast]^{-1})^{\binom{x}{1}}$ is a polynomial sequence with respect to the filtration $H_1\rtimes \mb{R} = H_1\rtimes \mb{R}\geqslant H_2\rtimes \{0\}\geqslant \cdots \geqslant H_s\rtimes \{0\} \geqslant \on{Id}_{H_1\rtimes \mb{R}}$. 

Note that pre- or post-multiplication by a fixed constant does not change $N$-periodicity. Furthermore note that $Z$ and all level $1$ and higher elements in the order given by removing all level $0$ elements gives a valid Mal'cev basis with respect to the proper filtration
\[H_1\rtimes \mb{R} = H_1\rtimes \mb{R}\geqslant H_2\rtimes \{0\}\geqslant \cdots \geqslant H_s\rtimes \{0\} \geqslant \on{Id}_{H_1\rtimes \mb{R}}.\]
Thus the polynomial sequence $([h_0]h_1^{\ast}[h_0]^{-1})^{\binom{x}{1}}$ with respect to the function $\wt{F}(\cdot)$ on nilmanifold $(H_1\rtimes\mb{R})/(\wt{\Gamma}\cap(H_1\rtimes\mb{R}))$ with this filtration and Mal'cev basis provides the desired construction.
\end{proof}

\section{Deferred lemmas}\label{sec:deferred}
We first prove the necessary partition of unity result on a nilmanifold. 
\begin{proof}[Proof of \cref{lem:nilmanifold-partition-of-unity}]
The key trick is to consider a ``smoothed sum'' of fundamental domains and perform a partition of unity. Let $\delta$ be a constant to be specified later and let $f(x)\ge 0$ be a smooth bump function on $\mb{R}$ with $\on{supp}(f)\in[-1/4,1/4]$, $\int_{\mb{R}}f(x)dx = 1$, and $\snorm{f}_{\infty}\le O(1)$. Let $H_j$ be smooth functions indexed by $j\in S$ of size $O(1/\delta)$ which are nonnegative with $\on{supp}(H_j)\in [j\delta,(j+2)\delta]$, $\snorm{H_j}_{\infty}\le O(1)$, and $\sum_{j\in S}H_j(x)$ equal to $1$ for $|x|\le 3/2$ and $0$ for $|x|\ge 2$.

For each $g\in G$ and $\beta\in\mb{R}$, there exists a unique $\gamma_i\in\Gamma$ such that
$\psi(g\gamma_i)\in[-\beta,1-\beta)^d$ (see e.g.~\cite[Lemma~B.6]{Len23b}). Define for $\beta\in\mb{R}$ the function $T_{\beta}\colon G\to\Gamma$ such that $\psi(gT_\beta(g))\in[-\beta,1-\beta)^d$; note that this function suffers from discontinuities at the boundaries of $[-\beta,1-\beta)^d$.

Note that for all $g\in G$ we have
\begin{align*}
1 &= \int_{\beta\in\mb{R}}f(\beta)d\beta=\int_{\beta\in \mb{R}}\prod_{i=1}^{d}\bigg(\sum_{j\in S}H_j(\psi_i(gT_{\beta}(g)))\bigg)f(\beta)d\beta\\
&=\sum_{j_1,\ldots,j_d\in S}\int_{\beta\in\mb{R}} \prod_{k=1}^{d}H_{j_k}(\psi_k(gT_{\beta}(g)))f(\beta)d\beta.
\end{align*}

The functions in our collection will be indexed by $j_1,\ldots,j_d$ and are defined by
\[\tau_{j_1,\ldots,j_d}(g\Gamma) = \int_{\beta\in\mb{R}}\prod_{k=1}^{d}H_{j_k}(\psi_k(gT_{\beta}(g)))f(\beta)d\beta.\]

This is a well-defined function since $gT_{\beta}(g) = g\gamma T_{\beta}(g\gamma)$ for all $\gamma\in\Gamma$. Furthermore note that $\snorm{\tau_{j_1,\ldots,j_d}}_{\infty}\le\exp(O(d))$. We will use these as our partition of unity.

Let $L = O(M)^{O_s(d^{O_s(1)})}$ with implicit constants chosen sufficiently large. To check the Lipschitz property, it suffices to consider $x\Gamma, y\Gamma$ such that $d_{G/\Gamma}(x\Gamma, y\Gamma)\le L^{-1}$. We may find $\wt{x},\wt{y}\in G$ such that
\[d_{G/\Gamma}(x\Gamma, y\Gamma) = d_G(\wt{x},\wt{y}),~\psi(\wt{x}),\psi(\wt{y})\in [-1,2)^{d},\text{ and } d_{G}(\wt{x},\on{id}_G) + d_{G}(\wt{y},\on{id}_G)\le L.\]

If $d_G(\gamma,\on{id}_G)\ge L^3$ then note that for $g\in G$ with $d_G(g,\on{id}_G)\le L$ we have 
\begin{align*}
d_{G}(g\gamma,\on{id}_G)&\ge d_{G}(g\gamma,g) - d_G(g,\on{id}_G)\ge L^{-1}d_{G}(\gamma,\on{id}_G) - d_G(g,\on{id}_G)\ge L^{-1}(L^3) - L\ge L^2/2>3,
\end{align*}
using quantitative approximate left-invariance of the metric (\cite[Lemma~B.4]{Len23b}). Therefore we have $d_{G}(T_\beta(\wt{x}),\on{id}_G) + d_{G}(T_\beta(\wt{y}),\on{id}_G)\le 2L^3$ for $|\beta|\le 2$. Let $\wt{\Gamma}\subseteq\Gamma$ be the set of all $\wt{\gamma}\in\Gamma$ such that $d_{G}(\wt{\gamma},\on{id}_G)\le 2L^3$ and note that $|\wt{\Gamma}|\le L^{O_s(d^{O_s(1)})}$.

We now show the desired properties of a partition of unity. Let $\beta$ be drawn from the probability distribution $f(\beta)$. Let $\eta:= d_{G}(\wt{x},\wt{y})$ and note that for $\gamma\in \wt{\Gamma}$ we have $\snorm{\psi(\wt{x}\gamma)-\psi(\wt{y}\gamma)}\le \eta\cdot L^{O_s(d^{O_s(1)})}=:\eta'$. Therefore
\begin{align*}
\mb{P}_{\beta}[T_{\beta}(\wt{x})\neq T_{\beta}(\wt{y})] &\le \sum_{\gamma\in \wt{\Gamma}}\mb{P}_{\beta}[\psi(\wt{x}\gamma)\in [-\beta,1-\beta)^{d} \cap \psi(\wt{y}\wt{x}^{-1}(\wt{x}\gamma))\notin [-\beta,1-\beta)^{d}]\\
&\le \sum_{\gamma\in \wt{\Gamma}}\sum_{i=1}^{d}\mb{P}_{\beta}[\psi_i(\wt{x}\gamma)\in [-\beta - \eta',-\beta]\cup [1-\beta-\eta',1-\beta)]\le\eta\delta^{-1}L^{O_s(d^{O_s(1)})}.
\end{align*}
Thus 
\begin{align*}
\bigg|\tau_{j_1,\ldots,j_d}(\wt{x}\Gamma) - \tau_{j_1,\ldots,j_d}(\wt{y}\Gamma)\bigg|&\le \int_{\beta\in \mb{R}}\bigg| \prod_{k=1}^{d}H_{j_k}(\psi(\wt{x}T_{\beta}(\wt{x}))) -  \prod_{k=1}^{d}H_{j_k}(\psi(\wt{y}T_{\beta}(\wt{x})))\bigg|f(\beta)d\beta \\
&\qquad\qquad+\delta^{-O(d)}\cdot \mb{P}_{\beta}[T_{\beta}(\wt{x})\neq T_{\beta}(\wt{y})]\\
&\le \delta^{-O(d)}\eta L^{O_s(d^{O_s(1)})},
\end{align*}
which demonstrates the necessary Lipschitz bound.

Finally we check the supports; note that in order for $\tau_{j_1,\ldots,j_d}(g\Gamma)$ to be nonzero, there exists $\gamma\in\Gamma$ such that $\psi(g\gamma) \in \prod_{k=1}^{d}[j_k\delta,(j_k+2)\delta]$. Let $\wt{g}$ be such that $\psi(\wt{g}) = ((j_k+1)\delta)_{1\le k\le d}\in[-2,2]^{d}$. We have $\snorm{\psi(g\gamma)-\psi(\wt{g})}\le\delta$. Thus, taking the fundamental domain for $G/\Gamma$ which is $\prod_{k=1}^{d}[(j_k+1)\delta-1/2,(j_k+1)\delta+1/2)$ (with respect to $\psi$) we have that the support is contained within a $\delta$-ball of the center. Note as $\psi_{\exp}\circ \psi^{-1}$ is a polynomial with coefficients bounded by $O(M)^{O_s(d^{O_s(1)})}$ and degree $O_s(1)$ (\cite[Lemma~B.1]{Len23b}) we have the desired result taking $\delta = \eps M^{-O_s(d^{O_s(1)})}$ (and appropriately modifying the definition of $L$).
\end{proof}

We next need that sufficiently divisible structure constants prove that $\psi_{\exp}^{-1}(\Gamma)$. 
\begin{proof}[Proof of \cref{lem:structure-constant}]
First, consider any element of $\Gamma$ which can be written $\gamma=\prod_{i=1}^{d}\exp(t_iX_i)$ for $t_i\in\mb{Z}$. We inductively prove that $\prod_{i=j}^d\exp(t_iX_i)$ is in $\psi_{\exp}^{-1}(\mb{Z}^d)$. Suppose we have shown that $\prod_{i=j+1}^{d}\exp(t_iX_i) = \exp(\sum_{i=j+1}^{d}u_iX_i)$ with $u_i\in\mb{Z}$, for some $1\le j\le d-1$ (the base case $j=d-1$ is obvious). Then
\[\prod_{i=j}^{d}\exp(t_iX_i)= \exp(t_jX_j)\exp\bigg(\sum_{i=j+1}^{d}u_iX_i\bigg) = \exp\bigg(t_jX_j + \sum_{i=j+1}^{d}u_iX_i + \cdots\bigg)\]
where the remainder is a finite list of commutators of $t_jX_j$ and $\sum_{i=j+1}^{d}u_iX_i$ coming from the Baker--Campbell--Hausdorff formula. As the coefficients are rationals with denominators bounded by $C_s$, it follows immediately from the condition on divisibility of Lie bracket structure constants that the inductive step holds. We deduce $\psi_{\exp}(\Gamma)\subseteq\mb{Z}^d$.

For the reverse inclusion, we also use induction. We show for all $1\le j\le d-1$ that for any $u_i\in\mb{Z}$, $\exp(\sum_{i=j}^du_iX_i)$ can be written in the form $\prod_{i=j}^d\exp(t_iX_i)$ for $t_i\in\mb{Z}$. Suppose we have the result for $j+1$, so that we have $\exp(\sum_{i=j+1}^{d}u_iX_i) = \prod_{i=j+1}^{d}\exp(t_iX_i)$. Then
\[\exp(-u_jX_j)\exp\bigg(\sum_{i=j}^{d}u_iX_i\bigg) = \exp\bigg(\sum_{i=j+1}^{d}u_i'X_i\bigg)\]
for $u_i'\in\mb{Z}$ using the Baker--Campbell--Hausdorff formula, the divisibility assumption, and the filtered nature of the Mal'cev basis. (The term on the far left is tailored to cancel the $X_j$ part.) Again, the inductive step immediately follows. This demonstrates $\psi_{\exp}^{-1}(\mb{Z}^d)\subseteq\Gamma$, and we are done.
\end{proof}

We next prove a comparison estimate between the distance in Mal'cev coordinates of the first kind and the associated torus metric. 

\begin{proof}[Proof of \cref{lem:first-to-torus}]
Note that 
\[d_{G/\Gamma}(x\Gamma,y\Gamma)\le d_{G}(x,y).\]
As $\psi\circ\psi_{\exp}^{-1}$ is a polynomial of degree $O_s(1)$ with coefficients bounded by $O(M)^{O_s(d^{O_s(1)})}$ (\cite[Lemma~B.1]{Len23b}) we have that 
\[\snorm{\psi(x)}_{\infty}\le (LM)^{O_s(d^{O_s(1)})}\text{ and }\snorm{\psi(x)-\psi(y)}_{\infty}\le \eps(LM)^{O_s(d^{O_s(1)})}.\]
The result then follows from \cite[Lemma~B.3]{Len23b}.
\end{proof}

We now give the short proof of \cref{lem:product-degree}.
\begin{proof}[Proof of \cref{lem:product-degree}]
It is evidently enough to show it for two functions $f_1,f_2$. Additionally, recall the product rule
\begin{equation}\label{eq:discrete-product-rule}
\partial_h(fg)=(\partial_hf)\cdot(T_hg)+f\cdot(\partial_hg)
\end{equation}
where $T_hg(x)=g(x+h)$ defines the translation operator. Note that all translation operators commute with all other operations such as discrete derivatives and products, and other translations. Also, this is valid at a point $x$ so long as $x,x+h$ are in the domains of $f,g$.

Let $s=d_1+d_2$. Suppose we are given $x,h_1,\ldots,h_{s+1}$ with $x+\{0,h_1\}+\cdots+\{0,h_{s+1}\}\subseteq S$. We have, by iterating \cref{eq:discrete-product-rule},
\begin{equation}\label{eq:discrete-product-rule-iterated}
\partial_{h_1,\ldots,h_{s+1}}(f_1f_2)(x)=\sum_{T_1\sqcup T_2=[s+1]}(\partial_{(h_j)_{j\in T_1}}f_1)(x)\cdot T_{\sum_{j\in T_1}h_j}(\partial_{(h_j)_{j\in T_2}}f_2)(x).
\end{equation}
This is seen to be valid since every element of the cube $x+\{0,h_1\}+\cdots+\{0,h_{s+1}\}$ is in the domain $S$. Now, for every disjoint partition $T_1\sqcup T_2=[s+1]$ we have either $|T_1|\ge d_1+1$ or $|T_2|\ge d_2+1$. By the given condition of $f_j$ being locally degree $d_j+1$, we see one of the two terms in \cref{eq:discrete-product-rule-iterated} is always $0$. The result follows.
\end{proof}

\end{document}